\documentclass[reqno]{amsart}
\usepackage{amssymb}

\newtheorem{theorem}{Theorem}[section]

\newtheorem{lemma}[theorem]{Lemma}

\newtheorem{cor}[theorem]{Corollary}
\theoremstyle{definition}
\newtheorem{definition}[theorem]{Definition}

\newtheorem{remark}[theorem]{Remark}

\allowdisplaybreaks
\numberwithin{equation}{section}
\begin{document}

\title[Diagonal hypersurfaces and elliptic curves over finite fields]
{Diagonal hypersurfaces and elliptic curves over finite fields and hypergeometric functions}


 \author{Sulakashna}
\address{Department of Mathematics, Indian Institute of Technology Guwahati, North Guwahati, Guwahati-781039, Assam, INDIA}
\curraddr{}
\email{sulakash@iitg.ac.in}
\author{Rupam Barman}
\address{Department of Mathematics, Indian Institute of Technology Guwahati, North Guwahati, Guwahati-781039, Assam, INDIA}
\curraddr{}
\email{rupam@iitg.ac.in}

\thanks{}


\subjclass[2010]{11G25, 33E50, 11S80, 11T24.}
\date{July 22, 2023, version-1}
\keywords{character sum; hypergeometric series; $p$-adic gamma function; diagonal hypersurfaces; elliptic curves.}
\begin{abstract} 
Let $D_\lambda^{d,k}$ denote the family of diagonal hypersurface over a finite field $\mathbb{F}_q$ given by
\begin{align*}
D_\lambda^{d,k}:X_1^d+X_2^d=\lambda dX_1^kx_2^{d-k},
\end{align*}
where $d\geq2$, $1\leq k\leq d-1$, and $\gcd(d,k)=1$. Let $\#D^{d,k}_\lambda$ denote the number of points on $D_\lambda^{d,k}$ in $\mathbb{P}^{1}(\mathbb{F}_q)$. It is easy to see that $\#D_\lambda^{d,k}$ is equal to the number of distinct zeros of the polynomial $y^d-d\lambda y^k+1\in \mathbb{F}_q[y]$ in $\mathbb{F}_q$. In this article, we prove that $\#D^{d,k}_\lambda$ is also equal to the number of distinct zeros of the polynomial $y^{d-k}(1-y)^k-(d\lambda)^{-d}$ in $\mathbb{F}_q$. We express the number of distinct zeros of the polynomial $y^{d-k}(1-y)^k-(d\lambda)^{-d}$ in terms of a $p$-adic hypergeometric function. Next, we derive summation identities for the $p$-adic hypergeometric functions appearing in the expressions for $\#D^{d,k}_\lambda$. Finally, as an application of the summation identities, we prove identities for the trace of Frobenius endomorphism on certain families of elliptic curves.
\end{abstract}
\maketitle
\section{Introduction and statement of results}
The arithmetic properties of Gauss and Jacobi sums have a very long history in number theory, with applications in Diophantine equations and the theory of $L$-functions. Number theorists have obtained generalizations of classical hypergeometric functions that are assembled with these sums, and these functions have significant applications in arithmetic geometry. Here we make use of these functions, as developed by Greene, McCarthy, and Ono \cite{greene, greene2,mccarthy2,mccarthy3,ono}.
Let $p$ be an odd prime, and let $\mathbb{F}_q$ be the finite field containing $q$ elements, where $q=p^r,r\geq1$. Let $\widehat{\mathbb{F}_q^{\times}}$ be the group of all the multiplicative
characters on $\mathbb{F}_q^{\times}$. We extend the domain of each $\chi\in \widehat{\mathbb{F}_q^{\times}}$ to $\mathbb{F}_q$ by setting $\chi(0):=0$
including the trivial character $\varepsilon$. Let $\delta$ denote the function on $\widehat{\mathbb{F}_q^{\times}}$ defined by
\begin{align}\label{delta}
\delta(A):=\left\{
	\begin{array}{ll}
		1, & \hbox{if $A$ is the trivial character;} \\
		0, & \hbox{otherwise.}
	\end{array}
	\right.
\end{align}
 Let $D_\lambda^d$ be the family of monomial deformation of diagonal hypersurfaces over $\mathbb{F}_q$ defined by the homogeneous equation
 \begin{align*}
D_\lambda^d: X_1^d+X_2^d+\cdots+X_n^d=\lambda d X_1^{h_1}X_2^{h_2}\cdots X_n^{h_n},
\end{align*} 
where $d,n\geq2$, $h_i\geq1$, $\sum_{i=1}^n h_i=d$, and $\gcd(d,h_1,h_2,\ldots, h_n)=1$. 
\par In this article, we study the family of diagonal hypersurfaces $D_\lambda^d$ with $n=2$. For $\lambda\in\mathbb{F}_q$, let $D_\lambda^{d,k}$ denote the family of diagonal hypersurfaces defined as
\begin{align}\label{eqn-0.1}
	D_\lambda^{d,k}:X_1^d+X_2^d=d\lambda X_1^k X_2^{d-k},
\end{align}
where $d\geq 2$, $k\geq 1$, $d>k$, and $\gcd(d,k,d-k)=1$. Note that $\gcd(d,k,d-k)=1$ if and only if $\gcd(d,k)=1$.
\subsection{Number of $\mathbb{F}_q$-points on Diagonal hypersurfaces}
Some of the biggest motivations for studying Gaussian and $p$-adic hypergeometric functions have been their connections with counting points on certain kinds of algebraic varieties and with Fourier coefficients and eigenvalues of modular forms. Let $\#D^{d,k}_\lambda$ denote the number of points on \eqref{eqn-0.1} in $\mathbb{P}^{1}(\mathbb{F}_q)$. In \cite{SB}, we expressed $\#D^{d,k}_\lambda$ in terms of special values of the $p$-adic hypergeometic function. Let ${_{n}}G_{n}\left[\cdots\right]$ be the $p$-adic hypergeometic function as introduced in Section 2. Then, the relation between $\#D^{d,k}_\lambda$ and the $p$-adic hypergeometic function is given by the following result:
\begin{theorem}\emph{(\cite[Theorem 1.3]{SB}).}\label{Theorem-1}
	Let $p$ be an odd prime and $q=p^r, \ r\geq 1$. Let $d\geq2$ and $k\geq 1$ be integers such that $d>k$ and $\gcd(d,k)=1$, and let $D_\lambda^{d,k}$ be the diagonal hypersurface given in \eqref{eqn-0.1} such that $p\nmid dk(d-k)$. 
	Then, for $\lambda \neq 0$, the number of points on $D_\lambda^{d,k}$ in $\mathbb{P}^{1}(\mathbb{F}_q)$ is given by
	\begin{align*}
		\#D_\lambda^{d,k} = 1 + {_{d-1}}G_{d-1}\left[\begin{array}{ccccccc}
			\frac{1}{d}, &\hspace{-.2cm} \frac{2}{d}, &\hspace{-.2cm} \ldots, &\hspace{-.2cm} \frac{k}{d}, &\hspace{-.2cm} \frac{k+1}{d}, &\hspace{-.2cm} \ldots, &\hspace{-.2cm} \frac{d-1}{d}\vspace{.1cm} \\
			0, &\hspace{-.2cm} \frac{1}{k}, &\hspace{-.2cm}  \ldots, &\hspace{-.2cm} \frac{k-1}{k}, &\hspace{-.2cm} \frac{1}{d-k}, &\hspace{-.2cm} \ldots, &\hspace{-.2cm} \frac{d-k-1}{d-k}
		\end{array}|\lambda^d k^{k} (d-k)^{d-k}
		\right]_q.
	\end{align*}
\end{theorem}
Let $N_q^d(\lambda)$ denote the number of points on the diagonal hypersurface $D_\lambda^{d,k}$ in $\mathbb{A}^{2}(\mathbb{F}_q)$. Then we have
	\begin{align}\label{neweqn-02}
	\# D_\lambda^{d,k}=\frac{N_q^d(\lambda) -1}{q-1}. 
	\end{align}
	Let $f(x_1,x_2):=x_1^d+x_2^d-d\lambda x_1^{k} x_2^{d-k}$. One can dehomogenize $f(x_1,x_2)$ to obtain $g_\lambda(y):=y^d-d\lambda y^k+1$. Also, if $\theta$ is a root of $g_\lambda(y)$, then $(\alpha\theta,\alpha)$ lies on $f(x_1,x_2)=0$ for all $\alpha\in\mathbb{F}_q^\times$. Suppose $g_\lambda(y)$ has $r_q'(\lambda)$ distinct zeros in $\mathbb{F}_q$. Then, $N_q^d(\lambda)=(q-1) r_q'(\lambda)+1$ and hence, 
	\begin{align}\label{neweqn-11}
	\#D^{d,k}_\lambda=r_q'(\lambda).
	\end{align}
	In this article, we relate $\#D^{d,k}_\lambda$ to the number of zeros of another polynomial over $\mathbb{F}_q$. In the following theorem, we prove that $\#D^{d,k}_\lambda$ is equal to the number of distinct zeros of the polynomial $f_\lambda(y)=y^{d-k}(1-y)^k-(d\lambda)^{-d}$. 
	\begin{theorem}\label{MT-2}
		Let $d>k\geq 1$ be integers such that $\gcd(k,d)=1$. Let $p$ be an odd prime such that $p\nmid dk(d-k)$. Let $q=p^r$, $r\geq1$. For $\lambda\in \mathbb{F}^{\times}_q$, let $f_\lambda(y)=y^{d-k}(1-y)^k-(d\lambda)^{-d}\in \mathbb{F}_q[y]$, and let $r_q(\lambda)$ be the number of distinct zeros of $f_\lambda(y)$ in $\mathbb{F}_q$. Let $\#D_\lambda^{d,k}$ be the number of points on $D_\lambda^{d,k}$ in $\mathbb{P}^1(\mathbb{F}_q)$. Then, we have $r_q(\lambda)=\#D_\lambda^{d,k}$.
	\end{theorem}
If $k=1$, then $g_\lambda(y)$ can be obtained from $f_\lambda(y)$ by a change of the variable.
However, for $k>1$, we do not know if $g_\lambda(y)$ can be obtained from $f_\lambda(y)$ in a similar way. Combining \eqref{neweqn-11} and Theorem \ref{MT-2}, we obtain the following corollary.
\begin{cor}\label{cor-0}
	Let $d>k\geq 1$ be integers. Let $p$ be an odd prime such that $p\nmid dk(d-k)$. Let $q=p^r$, $r\geq1$. For $\lambda\in \mathbb{F}_q^\times$, let $r_q(\lambda)$ and $r_q'(\lambda)$ be the number of distinct zeros of $f_{\lambda}(y)=y^{d-k}(1-y)^k-(d\lambda)^{-d}$ and $g_\lambda(y)=y^d-d\lambda y^k+1$ in $\mathbb{F}_q$, respectively. If $\gcd(d,k)=1$, then we have $r_q'(\lambda)=r_q(\lambda)$.
\end{cor}
Theorem \ref{MT-2} relates $r_q(\lambda)$, the number of distinct zeros of the polynomial $f_{\lambda}(y)=y^{d-k}(1-y)^k-(d\lambda)^{-d}$ in $\mathbb{F}_q$, to $\#D_\lambda^{d,k}$ for integers $d>k\geq 1$ with $\gcd(k,d)=1$. In the following theorem, we express $r_q(\lambda)$ in terms of the $p$-adic hypergeometric functions without the condition $\gcd(k,d)=1$. To be specific, we prove the following theorem.
\begin{theorem}\label{MT-1}
Let $d>k\geq 1$ be integers. Let $p$ be an odd prime such that $p\nmid dk(d-k)$. Let $q=p^r$, $r\geq1$. For $\lambda\in\mathbb{F}_q^\times$, let $r_q(\lambda)$ denote the number of distinct zeros of $f_{\lambda}(y)=y^{d-k}(1-y)^k-(d\lambda)^{-d}\in \mathbb{F}_q[y]$ in $\mathbb{F}_q$. We have 
		\begin{align*}
		r_q(\lambda)
		&=1+ {_{d-1}}G_{d-1}\left[\begin{array}{ccccccc}
		\frac{1}{d}, \hspace{-.1cm}&\frac{2}{d}, \hspace{-.1cm}&  \ldots, \hspace{-.1cm} & \frac{k}{d}, \hspace{-.1cm}& \frac{k+1}{d}, \hspace{-.1cm}& \ldots, \hspace{-.1cm}& \frac{d-1}{d}\vspace{0.1cm} \\
		0, \hspace{-.1cm}& \frac{1}{k}, \hspace{-.1cm}& \ldots,  \hspace{-.1cm}& \frac{k-1}{k}, \hspace{-.1cm} & \frac{1}{d-k},  \hspace{-.1cm}& \ldots, \hspace{-.1cm}& \frac{d-k-1}{d-k} 
		\end{array}|k^k(d-k)^{d-k}\lambda^d
		\right]_q\\
		&-\frac{1-q}{q}\hspace{-.3cm}\sum_{\begin{subarray}{1} \ \chi\in\widehat{\mathbb{F}_q^{\times}}, \\ \chi^{k_1}= \varepsilon,\chi\neq\varepsilon \end{subarray}}\overline{\chi}^d(-\lambda d)\delta(\chi^d),
		\end{align*}
		where $k_1=\gcd(q-1,k)$.
	\end{theorem} 
\begin{remark}
We remark that Theorem \ref{Theorem-1} follows from Theorem \ref{MT-2} and Theorem \ref{MT-1}. Suppose that $\gcd(k,d)=1$. Then, we have $\gcd(d,k_1)=1$. Now, if $\chi^{k_1}=\varepsilon$ and $\chi^d=\varepsilon$ then the order of $\chi$ divides $\gcd(k_1,d)$. Since $\gcd(k_1,d)=1$, we have $\chi=\varepsilon$. Therefore, the last summation in Theorem \ref{MT-1} is empty, and hence
\begin{align*}
r_q(\lambda)=1+{_{d-1}}G_{d-1}\left[\begin{array}{ccccccc}
\frac{1}{d}, \hspace{-.1cm}&\frac{2}{d}, \hspace{-.1cm}&  \ldots,  \hspace{-.1cm}& \frac{k}{d}, \hspace{-.1cm}& \frac{k+1}{d}, \hspace{-.1cm}& \ldots, \hspace{-.1cm}& \frac{d-1}{d}\vspace{0.1cm} \\
0, \hspace{-.1cm}& \frac{1}{k}, \hspace{-.1cm}& \ldots,  \hspace{-.1cm}& \frac{k-1}{k},  \hspace{-.1cm}& \frac{1}{d-k},  \hspace{-.1cm}& \ldots, \hspace{-.1cm}& \frac{d-k-1}{d-k} 
\end{array}|k^k(d-k)^{d-k}\lambda^d
\right]_q.
\end{align*}
Now, by Theorem \ref{MT-2}, we readily obtain Theorem \ref{Theorem-1}.
\end{remark}
\subsection{Summation identities and applications to elliptic curves}
In view of significant presence of Gaussian and $p$-adic hypergeometric functions in arithmetic geometry, it is an interesting problem to find transformation formulas and special values of these hypergeometric functions. In recent times, several transformation formulas and special values of the $p$-adic hypergeometric functions were found, see for example \cite{BS2,BS1,BS3,BS4,BSM,NS,NS1,SB1}. 
In this article, we prove summation identities for the $p$-adic hypergeometric functions appearing in the expressions for the number of $\mathbb{F}_q$-points on the diagonal hypersurface $D_\lambda^{d,k}$. As an application of the summation identities, we prove identities for the trace of Frobenius endomorphism on certain families of elliptic curves and $p$-adic hypergeometric functions.
\par Firstly, we state the summation identities. Let $\varphi$ denote the quadratic character on $\mathbb{F}_q$.
\begin{theorem}\label{MT-3}
	Let $d>k\geq 1$ be odd integers. Let $p$ be an odd prime such that $p\nmid dk(d-k)$. Let $q=p^r, r\geq1$. Then, for $x\in\mathbb{F}_q^{\times}$, we have 
	\begin{align*}
		&1+q\cdot{_{d-1}}G_{d-1}\left[\begin{array}{ccccccc}
			\frac{1}{d}, & \frac{2}{d}, & \ldots, & \frac{k}{d}, & \frac{k+1}{d}, & \ldots , & \frac{d-1}{d}\vspace{0.1cm} \\
			0,& \frac{1}{k}, & \ldots, & \frac{k-1}{k}, & \frac{1}{d-k},  & \ldots , & \frac{d-k-1}{d-k} 
		\end{array}|x
		\right]_q	\\
		&=-\sum_{t\in\mathbb{F}_q}\varphi(t(t-1))\times \nonumber\\
		&{_{d-1}}G_{d-1}\left[\begin{array}{cccccccccc}
			\frac{1}{d}, \hspace{-.2cm} & \ldots, \hspace{-.2cm}& \frac{k}{d}, \hspace{-.2cm}& \frac{k+1}{d}, \hspace{-.2cm}& \ldots, \hspace{-.2cm}& \frac{k+\frac{d-k}{2}-1}{d}, \hspace{-.2cm}& \frac{k+\frac{d-k}{2}}{d}, \hspace{-.2cm}& \ldots, \hspace{-.2cm}& \frac{d-2}{d}, \hspace{-.2cm}& \frac{d-1}{d} \vspace{0.1cm}\\
			0, \hspace{-.2cm}& \ldots, \hspace{-.2cm} & \frac{k-1}{k}, \hspace{-.2cm} & \frac{1}{d-k}, \hspace{-.2cm}& \ldots, \hspace{-.2cm}& \frac{\frac{d-k}{2}-1}{d-k}, \hspace{-.2cm}& \frac{\frac{d-k}{2}+1}{d-k}, \hspace{-.2cm}& \ldots, \hspace{-.2cm}& \frac{d-k-1}{d-k}, \hspace{-.2cm}& 0 
		\end{array}|xt
		\right]_q.
	\end{align*} 
\end{theorem}
\begin{theorem}\label{MT-4}
Let $d>k\geq 2$ be integers with $d$ even. Let $p$ be an odd prime such that $p\nmid dk(d-k)$. Let $q=p^r, r\geq1$. Then, for $x\in\mathbb{F}_q$, we have 
	\begin{align*}
		&\sum\limits_{t\in\mathbb{F}_q}\varphi(1-t){_{d-2}}G_{d-2}\left[\begin{array}{cccccc}
			\frac{1}{d}, & \ldots, & \frac{\frac{d}{2}-1}{d}, & \frac{\frac{d}{2}+1}{d},& \ldots,  & \frac{d-1}{d} \\
			c_1, & \ldots, & c_{\frac{d}{2}-1}, & c_{\frac{d}{2}}, & \ldots, & c_{d-2} 
		\end{array}|xt
		\right]_q\\
		&=-{_{d-1}}G_{d-1}\left[\begin{array}{cccccccc}
			\frac{1}{d}, & \frac{2}{d}, & \ldots, & \frac{k}{d}, & \frac{k+1}{d}, & \ldots , & \frac{d-1}{d} \vspace{0.1cm}\\
			0,& \frac{1}{k}, & \ldots, & \frac{k-1}{k}, & \frac{1}{d-k},  & \ldots , & \frac{d-k-1}{d-k}  
		\end{array}|x
		\right]_q,	
	\end{align*} 
	where  $\{c_1,\ldots,c_{d-2}\}=\{\frac{1}{k},\ldots,\frac{k-1}{k}, \frac{1}{d-k},\ldots, \frac{d-k-1}{d-k}\}$.
\end{theorem}
For example, if we take $d=5, k=3$ and $d=6, k=3$ in Theorems \ref{MT-3} and \ref{MT-4}, respectively, then for $p>5$ and $x\in\mathbb{F}_q^\times$, we obtain the following identities.
\begin{align*}
	-\sum_{t\in\mathbb{F}_q}\varphi(t(t-1))
	{_4}G_{4}\left[\begin{array}{cccc}
		\frac{1}{5}, &  \frac{2}{5}, & \frac{3}{5}, & \frac{4}{5}\vspace*{.05cm} \\
		0, & 0, & \frac{1}{3}, &  \frac{2}{3}
	\end{array}|xt
	\right]_q &=1+q\cdot{_4}G_{4}\left[\begin{array}{cccc}
		\frac{1}{5}, & \frac{2}{5}, & \frac{3}{5}, & \frac{4}{5}\vspace*{.05cm}\\
		\frac{1}{3}, & \frac{2}{3}, & 0,  & \frac{1}{2}
	\end{array}|x
	\right]_q,	\\
	\sum\limits_{t\in\mathbb{F}_q}\varphi(1-t){_{4}}G_{4}\left[\begin{array}{cccc}
		\frac{1}{6}, & \frac{2}{6}, & \frac{4}{6}, & \frac{5}{6} \vspace*{.05cm}\\
		\frac{1}{3}, & \frac{2}{3}, &  \frac{1}{3}, & \frac{2}{3}
	\end{array}|xt
	\right]_q
	&=-{_{5}}G_{5}\left[\begin{array}{ccccc}
		\frac{1}{6},& \frac{2}{6}, & \frac{3}{6},& \frac{4}{6}, & \frac{5}{6}\vspace*{.05cm}  \\
		\frac{1}{3}, &  \frac{2}{3}, & 0, & \frac{1}{3}, &  \frac{2}{3}
	\end{array}|x
	\right]_q.	
\end{align*}
For $r=1$, taking $k=1$ in Theorem \ref{MT-3} and $k=d-1$  in Theorem \ref{MT-4}, we obtain \cite[Theorem $1.2$]{BS5} and \cite[Theorem $1.3$]{BS5}, respectively.
\par 
There are many significant relations of hypergeometric functions to elliptic curves over finite fields. For example, see \cite{BK, BK1, fuselier,koike,lennon,lennon2,ono}. In \cite{mccarthy2}, McCarthy expressed the trace of the Frobenius endomorphism on elliptic curves in terms of special values of the $p$-adic hypergeometric function. Later, the second author with Saikia \cite{BS1}, gave another expression for the trace of the Frobenius endomorphism on elliptic curves and $p$-adic hypergeometric functions. Let $E$ be an elliptic curve given in the Weierstrass form over the finite field $\mathbb{F}_q$. Then the trace of Frobenius $a_q(E)$ of $E$ is given by
\begin{align*}
	a_q(E):=q+1-\#E(\mathbb{F}_q),
\end{align*}
where $\#E(\mathbb{F}_q)$ denotes the number of $\mathbb{F}_q$-points on $E$ including the point at infinity. Let $j(E)$ denote the $j$-invariant of the elliptic curve $E$. We now state two identities for $a_q(E)$ and $p$-adic hypergeometric functions.
\begin{theorem}\label{MT-5}
	Let $p>3$ be a prime and $q=p^r,r\geq1$ such that $q\not\equiv1\pmod3$. For $b,t\in\mathbb{F}_q^\times$, consider the elliptic curve $E_{t,b}:y^2=x^3+tx+b$ over $\mathbb{F}_q$. Then we have
	\begin{align*}
		\sum_{t\in\mathbb{F}_q^\times}\varphi(t(t^3-1))a_q(E_{t,b})=-q\cdot\varphi(b)\cdot{_3}G_{3}\left[\begin{array}{ccc}
			\frac{1}{4},  & \frac{1}{2},& \frac{3}{4} \vspace*{0.1cm}\\
			0, &	\frac{1}{3}, & \frac{2}{3}
		\end{array}|\frac{-27b^2}{4}
		\right]_q.
	\end{align*}
\end{theorem}
\begin{theorem}\label{MT-7}
	Let $p\geq3$ be a prime and $q=p^r,r\geq1$. For $f,t\in\mathbb{F}_q^\times$, consider the elliptic curve $E_{t,f}: y^2=x^3+fx^2+\frac{x}{t}$ over $\mathbb{F}_q$. Then we have 
	\begin{align*}
		&\sum_{t\in\mathbb{F}_q^\times}\varphi(t(t-1))a_q(E_{t,f})\\
		&=\left\{
		\begin{array}{lll}
			-\varphi(f)-q\cdot\varphi(2f), &\hbox{if $f^2=4$;}\\
			-\varphi(f), & \hbox{if $f^2-4$ is not a square;} \\
			-\varphi(f)-q\cdot\varphi(2f)(\varphi(1+\frac{a}{f})+\varphi(1-\frac{a}{f})), & \hbox{ if $f^2-4=a^2$.}
		\end{array}
		\right.
	\end{align*}
\end{theorem}
Next, we recall the Hessian form of elliptic curve over $\mathbb{F}_q$. Let $a\in\mathbb{F}_q$ be such that $a^3\neq1$. Then the Hessian curve over $\mathbb{F}_q$ is given by the following cubic equation
\begin{align*}
C_a(\mathbb{\mathbb{F}}_q):x^3+y^3+1=3axy.
\end{align*}
Let $\#C_a(\mathbb{F}_q)$ denote the number of $\mathbb{F}_q$-points on $C_a(\mathbb{F}_q)$. In the following theorem, we prove an identity for $\#C_a(\mathbb{F}_q)$.
\begin{theorem}\label{MT-6}
	Let $p>3$ be a prime and $q=p^r,r\geq1$ such that $q\not\equiv1\pmod3$. For $t\in \mathbb{F}_q$, $t\neq 0, 1$, consider the Hessian curve $C_t: x^3+y^3+1=3txy$ over $\mathbb{F}_q$. If $\#C_t(\mathbb{F}_q)$ denotes the number of $\mathbb{F}_q$-points on $C_t(\mathbb{F}_q)$, then we have
	\begin{align*}
		\sum_{t\in\mathbb{F}_q^\times,t\neq1 }\varphi(t(t^3-1))\#C_t(\mathbb{F}_q)=1.
	\end{align*}
\end{theorem}
The rest of this paper is organized as follows. In Section 2 we recall some basic properties of multiplicative characters, Gauss sums, and the $p$-adic gamma function. We recall hypergeometric functions over finite fields as introduced by Greene and McCarthy. We then define McCarthy's $p$-adic hypergeometric function. We also recall a few product formulas for the $p$-adic gamma function which are used to prove our main results. In Section 3 we prove Theorem \ref{MT-2} and Theorem \ref{MT-1}. In Section 4 we prove the summation identities for the $p$-adic hypergeometric function, namely Theorem \ref{MT-3} and Theorem \ref{MT-4}. We finally prove the identities for the trace of Frobenius of elliptic curves and $p$-adic hypergeometric functions in Section 5 as an application of our main results.
\section{Preliminaries}
Let $p$ be an odd prime, and let $\mathbb{F}_q$ be the finite field containing $q$ elements, where $q=p^r,r\geq1$. For multiplicative characters $A$ and $B$ on $\mathbb{F}_q$,
the binomial coefficient ${A \choose B}$ is defined by
\begin{align}\label{eq-0}
	{A \choose B}:=\frac{B(-1)}{q}J(A,\overline{B})=\frac{B(-1)}{q}\sum_{x \in \mathbb{F}_q}A(x)\overline{B}(1-x),
\end{align}
where $J(A, B)$ denotes the Jacobi sum and $\overline{B}$ is the character inverse of $B$. It is easy to see that the Jacobi sum satisfies the following identity:
\begin{align}\label{eq-1}
	J(A,B)=A(-1)J(A,\overline{AB}).
\end{align}
We recall the following properties of the binomial coefficients from \cite{greene}:
\begin{align}\label{eq-4}
	{A\choose \varepsilon}={A\choose A}=\frac{-1}{q}+\frac{q-1}{q}\delta(A),
\end{align}
where $\delta(\cdot)$ is the function as defined in \eqref{delta}.
We will make use of the following orthogonality relation for characters:
\begin{align}\label{eq-2}
	\sum_{\chi\in \widehat{\mathbb{F}_q^{\times}}} \chi(x)= \left\{
	\begin{array}{ll}
		q-1, & \hbox{if $x=1$;} \\
		0, & \hbox{otherwise.}
	\end{array}
	\right.
\end{align}
\par
Let $\mathbb{Z}_p$ and $\mathbb{Q}_p$ denote the ring of $p$-adic integers and the field of $p$-adic numbers, respectively.
Let $\overline{\mathbb{Q}_p}$ be the algebraic closure of $\mathbb{Q}_p$ and $\mathbb{C}_p$ be the completion of $\overline{\mathbb{Q}_p}$.
Let $\mathbb{Z}_q$ be the ring of integers in the unique unramified extension of $\mathbb{Q}_p$ with residue field $\mathbb{F}_q$.
We know that $\chi\in \widehat{\mathbb{F}_q^{\times}}$ takes values in $\mu_{q-1}$, where $\mu_{q-1}$ is the group of all the $(q-1)$-th roots of unity in $\mathbb{C}^{\times}$. Since $\mathbb{Z}_q^{\times}$ contains all the $(q-1)$-th roots of unity,
we can consider multiplicative characters on $\mathbb{F}_q^\times$
to be maps $\chi: \mathbb{F}_q^{\times} \rightarrow \mathbb{Z}_q^{\times}$.
Let $\omega: \mathbb{F}_q^\times \rightarrow \mathbb{Z}_q^{\times}$ be the Teichm\"{u}ller character.
For $a\in\mathbb{F}_q^\times$, the value $\omega(a)$ is just the $(q-1)$-th root of unity in $\mathbb{Z}_q$ such that $\omega(a)\equiv a \pmod{p}$.
\par Next, we introduce the Gauss sum and recall some of its elementary properties. For further details, see \cite{evans}. Let $\zeta_p$ be a fixed primitive $p$-th root of unity
in $\overline{\mathbb{Q}_p}$. The trace map $\text{tr}: \mathbb{F}_q \rightarrow \mathbb{F}_p$ is given by
\begin{align}
	\text{tr}(\alpha)=\alpha + \alpha^p + \alpha^{p^2}+ \cdots + \alpha^{p^{r-1}}.\notag
\end{align}
Then the additive character
$\theta: \mathbb{F}_q \rightarrow \mathbb{Q}_p(\zeta_p)$ is defined by
\begin{align}
	\theta(\alpha)=\zeta_p^{\text{tr}(\alpha)}.\notag
\end{align}
For $\chi \in \widehat{\mathbb{F}_q^\times}$, the \emph{Gauss sum} is defined by
\begin{align}
	g(\chi):=\sum\limits_{x\in \mathbb{F}_q}\chi(x)\theta(x) .\notag
\end{align}
We let $T$ denote a generator of $\widehat{\mathbb{F}_q^{\times}}$.
\begin{lemma}\emph{(\cite[Lemma 2.2]{fuselier}).}\label{lemma2_02} For $\alpha\in \mathbb{F}^{\times}_q$, we have
	\begin{align}
		\theta(\alpha)=\frac{1}{q-1}\sum_{m=0}^{q-2} g(T^{-m})T^m(\alpha).\notag
	\end{align}
\end{lemma}
\begin{lemma}\emph{(\cite[Eq. 1.12]{greene}).}\label{lemma2_1}
	For $\chi \in \widehat{\mathbb{F}_q^\times}$, we have
	$$g(\chi)g(\overline{\chi})=q\cdot \chi(-1)-(q-1)\delta(\chi).$$
\end{lemma}
The following lemma gives a relation between Jacobi and Gauss sums.
\begin{lemma}\emph{(\cite[Eq. 1.14]{greene}).}\label{lemma2_2} For $A,B\in\widehat{\mathbb{F}_q^{\times}}$, we have
	\begin{align}
		J(A,B)=\frac{g(A)g(B)}{g(AB)}+(q-1)B(-1)\delta(AB).\notag
	\end{align}
\end{lemma}
In \cite{greene, greene2}, Greene introduced the notion of hypergeometric series over finite fields famously known as \emph{Gaussian hypergeometric series}. He defined hypergeometric functions over finite fields using binomial coefficients as follows.
\begin{definition}(\emph{\cite[Definition 3.10]{greene}}).
	Let $n$ be a positive integer and $x\in \mathbb{F}_{q}$. For multiplicative characters $A_{0},A_{1},\dots,A_{n},B_{1}, \dots,B_{n}$ on $\mathbb{F}_{q}$, the ${_{n+1}}F_n$-hypergeometric function over $\mathbb{F}_{q}$ is defined by
	\begin{align*}
		&_{n+1}F_n\left(\begin{array}{cccc}
			A_0, & A_1, &  \ldots, & A_n \\
			& B_1, & \ldots, & B_n
		\end{array}|x
		\right)_q:=\frac{q}{q-1}\sum_{\chi\in\widehat{\mathbb{F}_{q}^{\times}}} {A_0 \chi \choose \chi}{A_{1}\chi \choose B_1\chi}\cdots {A_{n}\chi \choose B_n \chi}\chi(x).
	\end{align*}
\end{definition}
In \cite{mccarthy3}, McCarthy gave another definition of hypergeometric function over finite fields using the Gauss sums.
\begin{definition}(\emph{\cite[Definition 1.4]{mccarthy3}}).
	Let $n$ be a positive integer. For $x\in \mathbb{F}_{q}$ and multiplicative characters $A_{0},A_{1},\dots,A_{n},B_{1}, \dots,B_{n}$ on $\mathbb{F}_{q}$, the ${{_{n+1}}F_n }^*$-hypergeometric function over $\mathbb{F}_{q}$ is defined by
	\begin{align*}
		&_{n+1}F_n\left(\begin{array}{cccc}
			A_0, & A_1, &  \ldots, & A_n \\
			& B_1, & \ldots, & B_n
		\end{array}|x
		\right)_q^\ast:=\frac{-1}{q-1}\sum_{\chi\in\widehat{\mathbb{F}_{q}^{\times}}} \prod_{i=0}^{n}\frac{g(A_i\chi)}{g(A_i)}\\
		&\hspace{6cm} \times \prod_{j=1}^n \frac{g(\overline{B_j\chi})}{g(\overline{B_j})}g(\overline{\chi})\chi(-1)^{n+1}\chi(x).
	\end{align*}
\end{definition}
Now, we recall the $p$-adic gamma function. For further details, see \cite{kob}.
For a positive integer $n$,
the $p$-adic gamma function $\Gamma_p(n)$ is defined as
\begin{align}
	\Gamma_p(n):=(-1)^n\prod\limits_{0<j<n,p\nmid j}j\notag
\end{align}
and one extends it to all $x\in\mathbb{Z}_p$ by setting $\Gamma_p(0):=1$ and
\begin{align}
	\Gamma_p(x):=\lim_{x_n\rightarrow x}\Gamma_p(x_n)\notag
\end{align}
for $x\neq0$, where $x_n$ runs through any sequence of positive integers $p$-adically approaching $x$.
This limit exists, is independent of how $x_n$ approaches $x$,
and determines a continuous function on $\mathbb{Z}_p$ with values in $\mathbb{Z}_p^{\times}$.
Let $\pi \in \mathbb{C}_p$ be the fixed root of $x^{p-1} + p=0$ which satisfies
$\pi \equiv \zeta_p-1 \pmod{(\zeta_p-1)^2}$. Then the Gross-Koblitz formula relates Gauss sums and the $p$-adic gamma function as follows.
\begin{theorem}\emph{(\cite[Gross-Koblitz]{gross}).}\label{thm2_3} For $a\in \mathbb{Z}$ and $q=p^r, r\geq 1$, we have
	\begin{align}
		g(\overline{\omega}^a)=-\pi^{(p-1)\sum\limits_{i=0}^{r-1}\langle\frac{ap^i}{q-1} \rangle}\prod\limits_{i=0}^{r-1}\Gamma_p\left(\left\langle \frac{ap^i}{q-1} \right\rangle\right).\notag
	\end{align}
\end{theorem}
Next, we recall McCarthy's $p$-adic hypergeometric function. For $x \in \mathbb{Q}$, we let $\lfloor x\rfloor$ denote the greatest integer less than or equal to $x$ and $\langle x\rangle$ 
denote the fractional part of $x$, i.e., $x-\lfloor x\rfloor$, satisfying $0\leq\langle x\rangle<1$. McCarthy's $p$-adic hypergeometric function $_{n}G_{n}[\cdots]_q$ is defined as follows.
\begin{definition}(\emph{\cite[Definition 5.1]{mccarthy2}}). \label{defin1}
	Let $p$ be an odd prime and $q=p^r$, $r\geq 1$. Let $t \in \mathbb{F}_q$.
	For positive integer $n$ and $1\leq k\leq n$, let $a_k$, $b_k$ $\in \mathbb{Q}\cap \mathbb{Z}_p$.
	Then the function $_{n}G_{n}[\cdots]_q$ is defined by
	\begin{align}
		&_nG_n\left[\begin{array}{cccc}
			a_1, & a_2, & \ldots, & a_n \\
			b_1, & b_2, & \ldots, & b_n
		\end{array}|t
		\right]_q\notag\\
		&\hspace{1cm}:=\frac{-1}{q-1}\sum_{a=0}^{q-2}(-1)^{an}~~\overline{\omega}^a(t)
		\prod\limits_{k=1}^n\prod\limits_{i=0}^{r-1}(-p)^{-\lfloor \langle a_kp^i \rangle-\frac{ap^i}{q-1} \rfloor -\lfloor\langle -b_kp^i \rangle +\frac{ap^i}{q-1}\rfloor}\notag\\
		&\hspace{2cm} \times \frac{\Gamma_p(\langle (a_k-\frac{a}{q-1})p^i\rangle)}{\Gamma_p(\langle a_kp^i \rangle)}
		\frac{\Gamma_p(\langle (-b_k+\frac{a}{q-1})p^i \rangle)}{\Gamma_p(\langle -b_kp^i \rangle)}.\notag
	\end{align}
\end{definition}
We recall some lemmas which relate certain products of values of the $p$-adic gamma function. We will use these lemmas in the proof of our main results.
\begin{lemma}\emph{(\cite[Lemma 3.1]{BS1}).}\label{lemma-3_1}
	Let $p$ be a prime and $q=p^r, r\geq 1$. For $0\leq a\leq q-2$ and $t\geq 1$ with $p\nmid t$, we have
	\begin{align}
		\omega(t^{-ta})\prod\limits_{i=0}^{r-1}\Gamma_p\left(\left\langle\frac{-tp^ia}{q-1}\right\rangle\right)
		\prod\limits_{h=1}^{t-1}\Gamma_p\left(\left\langle \frac{hp^i}{t}\right\rangle\right)
		=\prod\limits_{i=0}^{r-1}\prod\limits_{h=0}^{t-1}\Gamma_p\left(\left\langle\frac{p^i(1+h)}{t}-\frac{p^ia}{q-1}\right\rangle \right).\notag
	\end{align}
\end{lemma}
\begin{lemma}\emph{(\cite[Lemma 3.2]{BS1}).}\label{lemma-3_2}
	Let $p$ be a prime and $q=p^r, r\geq 1$. For $0\leq a\leq q-2$ and $t\geq 1$ with $p\nmid t$, we have
	\begin{align*}
		\omega(t^{ta})\prod\limits_{i=0}^{r-1}\Gamma_p\left(\left\langle\frac{tp^ia}{q-1}\right\rangle\right)
		\prod\limits_{h=1}^{t-1}\Gamma_p\left(\left\langle \frac{hp^i}{t}\right\rangle\right)
		=\prod\limits_{i=0}^{r-1}\prod\limits_{h=0}^{t-1}\Gamma_p\left(\left\langle\frac{p^i h}{t}+\frac{p^ia}{q-1}\right\rangle \right).\notag
	\end{align*}
\end{lemma}
\begin{lemma}\emph{(\cite[Lemma 3.4]{BSM})}.\label{lemma-3_5}
Let $p$ be an odd prime and $q=p^{r}, r\geq 1$. For $0<a\leq q-2$, we have
\begin{align}\label{eq-12}
\prod_{i=0}^{r-1} \Gamma_{p}\left(\left\langle\left(1-\frac{a}{q-1}\right)p^{i}\right\rangle\right)\Gamma_{p}\left(\left\langle\frac{ap^{i}}{q-1}\right\rangle\right) = (-1)^r \overline{\omega}^{a}(-1).
\end{align}
For $0\leq a\leq q-2$ such that $a\neq \frac{q-1}{2}$, we have
\begin{align}\label{eq-13}
\prod_{i=0}^{r-1} \frac{\Gamma_{p}(\langle(\frac{1}{2}-\frac{a}{q-1})p^{i}\rangle)\Gamma_{p}(\langle(\frac{1}{2}+\frac{a}{q-1})p^{i}\rangle)}{\Gamma_{p}(\langle\frac{p^{i}}{2}\rangle)\Gamma_{p}(\langle\frac{p^{i}}{2}\rangle)} = \overline{\omega}^{a}(-1).
\end{align}
\end{lemma}
Finally, we recall two lemmas relating fractional and integral parts of certain rational numbers.
\begin{lemma}\emph{(\cite[Lemma 2.6]{SB}).}\label{lemma-3_3}
	Let $p$ be an odd prime and $q=p^r, r\geq 1$. Let $d\geq2$ be an integer such that $p\nmid d$. Then, for $1\leq a\leq q-2$ and $0\leq i\leq r-1$, we have
	\begin{align*}
	\left\lfloor\frac{ap^i}{q-1}\right\rfloor +\left\lfloor\frac{-dap^i}{q-1}\right\rfloor = \sum_{h=1}^{d-1} \left\lfloor\left\langle\frac{hp^i}{d}\right\rangle-\frac{ap^i}{q-1}\right\rfloor -1.
	\end{align*}
\end{lemma}
\begin{lemma}\emph{(\cite[Lemma 2.7]{SB}).}\label{lemma-3_4}
	Let $p$ be an odd prime and $q=p^r, r\geq 1$. Let $l$ be a positive integer such that $p\nmid l$. Then, for $0\leq a\leq q-2$ and $0\leq i\leq r-1$, we have
	\begin{align*}
	\left\lfloor\frac{lap^i}{q-1}\right\rfloor = \sum_{h=0}^{l-1} \left\lfloor\left\langle\frac{-hp^i}{l}\right\rangle+\frac{ap^i}{q-1}\right\rfloor.
	\end{align*}
\end{lemma}
\section{Proof of Theorems \ref{MT-2} and \ref{MT-1}}
In this section, we prove Theorems \ref{MT-2} and \ref{MT-1}. We first prove Theorem \ref{MT-2}.
\begin{proof}[Proof of Theorem \ref{MT-2}]
		For $\alpha\in\mathbb{F}^{\times}_q$, we consider a character sum defined by
	\begin{align*}
		C(d,k,\alpha):=\sum_{\chi\in\widehat{\mathbb{F}_q^{\times}}}\frac{g(\chi^d)g(\overline{\chi}^{d-k})\chi(\alpha)}{g(\chi^k)}.
	\end{align*}
	We note that if $k_1=\gcd(k,q-1)$ and $k=k_1k_2$ such that $\gcd(k_2,q-1)=1$ and $\chi^k=\varepsilon$ for some $\chi\in\widehat{\mathbb{F}_q^\times}$, then $\chi^{k_1}=\varepsilon$.
	We first write $C(d,k,\alpha)$ in terms of the number of distinct zeros of the polynomial $h_\alpha(y)=y^{d-k}(1-y)^k-(-1)^{d-k}\alpha\in\mathbb{F}_q[y]$. Applying Lemma \ref{lemma2_2} and then using \eqref{eq-1}, we have
	\begin{align}\label{eqn-1.50}
		C(d,k,\alpha)&=\sum_{\chi\in\widehat{\mathbb{F}_q^{\times}}}\frac{g(\chi^d)g(\overline{\chi}^{d-k})\chi(\alpha)}{g(\chi^k)}\nonumber\\
		&=\sum_{\chi\in\widehat{\mathbb{F}_q^{\times}}}(J(\overline{\chi}^{d-k},\chi^d)\chi(\alpha)-(q-1)\chi^d(-1)\chi(\alpha)\delta(\chi^k))\nonumber\\
		&=\sum_{\chi\in\widehat{\mathbb{F}_q^{\times}}}\overline{\chi}^{d-k}(-1)J(\overline{\chi}^{d-k},\overline{\chi}^k)\chi(\alpha)-\sum_{\chi\in\widehat{\mathbb{F}_q^{\times}}}(q-1)\chi^d(-1)\chi(\alpha)\delta(\chi^k)\nonumber\\
		&=\sum_{\chi\in\widehat{\mathbb{F}_q^{\times}}}\chi^{d-k}(-1)J(\chi^{d-k},\chi^k)\overline{\chi}(\alpha)-\sum_{\begin{subarray}{1} \ \chi\in\widehat{\mathbb{F}_q^{\times}}, \\ \chi^{k_1}= \varepsilon \end{subarray}}(q-1)\chi((-1)^d\alpha)\nonumber\\
		&=\sum_{\chi\in\widehat{\mathbb{F}_q^{\times}}}\sum_{y\in\mathbb{F}_q}\chi\left(\frac{y^{d-k}(1-y)^k}{(-1)^{d-k}\alpha}\right)-\sum_{\begin{subarray}{1} \ \chi\in\widehat{\mathbb{F}_q^{\times}}, \\ \chi^{k_1}= \varepsilon \end{subarray}}(q-1)\chi((-1)^d\alpha).
	\end{align}
	By \eqref{eq-2}, the inner sum of the first summation of \eqref{eqn-1.50} is nonzero only if $h_\alpha(y)$ has a solution in $\mathbb{F}_q$. Let $n_q(\alpha)$ be the number of distinct zeros of $h_\alpha(y)$ in $\mathbb{F}_q$. Then
	\begin{align}\label{eqn-1.5}
		C(d,k,\alpha)=(q-1)n_q(\alpha)-(q-1)\sum_{\begin{subarray}{1} \ \chi\in\widehat{\mathbb{F}_q^{\times}}, \\ \chi^{k_1}= \varepsilon \end{subarray}}\chi((-1)^d\alpha).
	\end{align}
We can also write $C(d,k,\alpha)$ as 
\begin{align*}
	C(d,k,\alpha)&=\sum_{\chi\in\widehat{\mathbb{F}_q^{\times}}}\frac{g(\chi^d)g(\overline{\chi}^{d-k})g(\overline{\chi}^k)\chi(\alpha)}{g(\chi^k)g(\overline{\chi}^k)}\\
	&=\sum_{\begin{subarray}{1} \ \chi\in\widehat{\mathbb{F}_q^{\times}}, \\ \chi^{k_1}\neq \varepsilon \end{subarray}}\frac{g(\chi^d)g(\overline{\chi}^{d-k})g(\overline{\chi}^k)\chi(\alpha)}{g(\chi^k)g(\overline{\chi}^k)}+\sum_{\begin{subarray}{1} \ \chi\in\widehat{\mathbb{F}_q^{\times}}, \\ \chi^{k_1}= \varepsilon \end{subarray}}\frac{g(\chi^d)g(\overline{\chi}^{d-k})g(\overline{\chi}^k)\chi(\alpha)}{g(\chi^k)g(\overline{\chi}^k)}.
\end{align*}
Using Lemma \ref{lemma2_1} in the first summation and the fact that $g(\varepsilon)=-1$ in the second summation, we obtain
\begin{align*}
	C(d,k,\alpha)&=\frac{1}{q}\sum_{\begin{subarray}{1} \ \chi\in\widehat{\mathbb{F}_q^{\times}}, \\ \chi^{k_1}\neq \varepsilon \end{subarray}}g(\chi^d)g(\overline{\chi}^{d-k})g(\overline{\chi}^k)\chi((-1)^k\alpha)-\sum_{\begin{subarray}{1} \ \chi\in\widehat{\mathbb{F}_q^{\times}}, \\ \chi^{k_1}= \varepsilon \end{subarray}}g(\chi^d)g(\overline{\chi}^{d})\chi(\alpha).
\end{align*}
Adding and subtracting the terms in the first summation for $\chi\in\widehat{\mathbb{F}_q^{\times}}$ such that $\chi^{k_1}=\varepsilon$, we deduce that 
\begin{align}\label{eqn-1.0}
	C(d,k,\alpha)&=\frac{1}{q}\sum_{\chi\in\widehat{\mathbb{F}_q^{\times}}}g(\chi^d)g(\overline{\chi}^{d-k})g(\overline{\chi}^k)\chi((-1)^k\alpha)-\sum_{\begin{subarray}{1} \ \chi\in\widehat{\mathbb{F}_q^{\times}}, \\ \chi^{k_1}= \varepsilon \end{subarray}}g(\chi^d)g(\overline{\chi}^{d})\chi(\alpha)\notag\\
	&\hspace*{.4cm}+\frac{1}{q}\sum_{\begin{subarray}{1} \ \chi\in\widehat{\mathbb{F}_q^{\times}}, \\ \chi^{k_1}= \varepsilon \end{subarray}}g(\chi^d)g(\overline{\chi}^{d})\chi(\alpha)\notag\\
	&=A+\frac{1-q}{q}\sum_{\begin{subarray}{1} \ \chi\in\widehat{\mathbb{F}_q^{\times}}, \\ \chi^{k_1}= \varepsilon \end{subarray}}g(\chi^d)g(\overline{\chi}^{d})\chi(\alpha),
\end{align}
where $A=\frac{1}{q}\sum_{\chi\in\widehat{\mathbb{F}_q^{\times}}}g(\chi^d)g(\overline{\chi}^{d-k})g(\overline{\chi}^k)\chi((-1)^k\alpha)$.
Using Lemma \ref{lemma2_1}, we obtain
	\begin{align*}
		C(d,k,\alpha)=A-\frac{q-1}{q}\sum_{\begin{subarray}{1} \ \chi\in\widehat{\mathbb{F}_q^{\times}}, \\ \chi^{k_1}= \varepsilon \end{subarray}}q\chi((-1)^d \alpha)+\frac{(q-1)^2}{q}\sum_{\begin{subarray}{1} \ \chi\in\widehat{\mathbb{F}_q^{\times}}, \\ \chi^{k_1}= \varepsilon \end{subarray}}\chi(\alpha)\delta(\chi^d).
	\end{align*}
	Since $\gcd(d,k)=1$, so we have $\gcd(d,k_1)=1$. If $\chi^{k_1}=\varepsilon$ and $\chi^d=\varepsilon$ then the order of $\chi$ divides $\gcd(k_1,d)$. Since $\gcd(k_1,d)=1$, we have $\chi=\varepsilon$. Thus
	\begin{align}\label{eqn-1.7}
		A=C(d,k,\alpha)-\frac{(q-1)^2}{q}+(q-1)\sum_{\begin{subarray}{1} \ \chi\in\widehat{\mathbb{F}_q^{\times}}, \\ \chi^{k_1}= \varepsilon \end{subarray}}\chi((-1)^d \alpha).
	\end{align}
	Combining \eqref{eqn-1.5} and \eqref{eqn-1.7}, we deduce that
	\begin{align}\label{eqn-1.8}
		A=(q-1)n_q(\alpha)-\frac{(q-1)^2}{q}.
	\end{align}
Taking $\alpha=(-1)^{d+k}(\lambda d)^{-d}$, we have $h_\alpha(y)=f_\lambda(y)$. Hence, $r_q(\lambda)=n_q(\alpha)$. \\
	Let $f(x_1,x_2)=x_1^d+x_2^d-d\lambda x_1^{k} x_2^{d-k}$ and using the identity 
\begin{align*}
	\sum_{z\in \mathbb{F}_q} \theta(zf(x_1,x_2))= \left\{
	\begin{array}{ll}
		q, & \hbox{if $f(x_1,x_2)=0$;} \\
		0, & \hbox{if $f(x_1,x_2)\neq 0$,}
	\end{array}
	\right.
\end{align*}
we obtain
\begin{align}\label{eq-19}
	q\cdot N_q^d(\lambda)&= \sum_{z\in\mathbb{F}_q}\sum_{x_i \in\mathbb{F}_q}\theta(zf(x_1,x_2))\nonumber\\
	&=q^2 + \sum_{z,x_1,x_2 \in\mathbb{F}_q^{\times}}\theta(zf(x_1,x_2)) +\sum_{z\in\mathbb{F}_q^{\times}}\sum_{\begin{subarray}{1} \ \text{some } \\ \  x_i =0\end{subarray}} \theta(zf(x_1,x_2))\nonumber\\
	&=q^2+A_1+B_1,
\end{align}
where $A_1=\displaystyle \sum_{z,x_i \in\mathbb{F}_q^{\times}}\theta(zf(x_1,x_2))$ and $B_1=\displaystyle \sum_{z\in\mathbb{F}_q^{\times}}\sum_{\begin{subarray}{1} \ \text{some } \\ \  x_i =0\end{subarray}} \theta(zf(x_1,x_2))$. 
Using Lemma \ref{lemma2_02} and the orthogonality relation \eqref{eq-2}, we calculated $B_1$ in \cite{SB} and found that $B_1=1-q$. Combining \eqref{neweqn-02} and \eqref{eq-19}, we have
	\begin{align*}
		\#D_\lambda^{d,k}&=\frac{N_q^d(\lambda)-1}{q-1}=\frac{q-1+\frac{A_1}{q}-\frac{q-1}{q}}{q-1}\\
		&=\frac{q-1}{q}+\frac{A_1}{q(q-1)}.
	\end{align*}
In \cite{SB}, using Lemma \ref{lemma2_02} and the condition that $\gcd(d,k)=1$, we simplified the character sum $A_1$ and found that 
	\begin{align*}
	A_1&=\sum_{a=0}^{q-2}g(T^{-ka})g(T^{-(d-k)a})g(T^{da})T^{-da}(-d\lambda).
\end{align*}
For $\alpha=(-1)^{d+k}(\lambda d)^{-d}$, we observe that $A=\frac{A_1}{q}$. Hence,
	\begin{align*}
		\#D_\lambda^{d,k}&=\frac{q-1}{q}+\frac{A}{q-1}.
	\end{align*} 
	Using \eqref{eqn-1.8} and the fact that $r_q(\lambda)=n_q(\alpha)$, we complete the proof.
\end{proof}
Next, we prove Theorem \ref{MT-1}.
\begin{proof}[Proof of Theorem \ref{MT-1}]
We first recall that for $\alpha\in\mathbb{F}_q^\times$, we have
	\begin{align*}
	C(d,k,\alpha)=\sum_{\chi\in\widehat{\mathbb{F}_q^{\times}}}\frac{g(\chi^d)g(\overline{\chi}^{d-k})\chi(\alpha)}{g(\chi^k)}.
\end{align*}
We now write $C(d,k,\alpha)$ in terms of $p$-adic hypergeometric functions. From \eqref{eqn-1.0}, we have
	\begin{align}\label{neweqn-01}
	C(d,k,\alpha)=A+\frac{1-q}{q}\sum_{\begin{subarray}{1} \ \chi\in\widehat{\mathbb{F}_q^{\times}}, \\ \chi^{k_1}= \varepsilon \end{subarray}}g(\chi^d)g(\overline{\chi}^{d})\chi(\alpha),
\end{align}
where $A=\frac{1}{q}\sum_{\chi\in\widehat{\mathbb{F}_q^{\times}}}g(\chi^d)g(\overline{\chi}^{d-k})g(\overline{\chi}^k)\chi((-1)^k\alpha)$. Now, we simplify the expression for $A$. We have
\begin{align*}
	A=\frac{1}{q}\sum_{\chi\in\widehat{\mathbb{F}_q^{\times}}}g(\chi^d)g(\overline{\chi}^{d-k})g(\overline{\chi}^k)\chi((-1)^k\alpha).
\end{align*}
Replacing $\chi$ with $\omega^a$ and then applying Gross-Koblitz formula, we obtain
\begin{align*}
	A&=\frac{1}{q}\sum_{a=0}^{q-2}g(\overline{\omega}^{-ad})g(\overline{\omega}^{a(d-k)})g(\overline{\omega}^{ak})\overline{\omega}^a((-1)^k\alpha^{-1})\\
	&=-\frac{1}{q}\sum_{a=0}^{q-2}(-p)^{\sum_{i=0}^{r-1}\left(\left\langle\frac{-adp^i}{q-1}\right\rangle+\left\langle\frac{a(d-k)p^i}{q-1}\right\rangle+\left\langle\frac{akp^i}{q-1}\right\rangle\right)}\overline{\omega}^a((-1)^k\alpha^{-1})\\
	&\times\prod_{i=0}^{r-1}\Gamma_p\left(\left\langle\frac{-adp^i}{q-1}\right\rangle\right)\Gamma_p\left(\left\langle\frac{akp^i}{q-1}\right\rangle\right)\Gamma_p\left(\left\langle\frac{a(d-k)p^i}{q-1}\right\rangle\right).
\end{align*}
Applying Lemma \ref{lemma-3_1} with $t=d$ and Lemma \ref{lemma-3_2} with $t=k$ and $t=d-k$, we deduce that
\begin{align*}
	A&=-\frac{1}{q}\sum_{a=0}^{q-2}(-p)^{\sum_{i=0}^{r-1}\left(-\left\lfloor\frac{-adp^i}{q-1}\right\rfloor-\left\lfloor\frac{akp^i}{q-1}\right\rfloor-\left\lfloor\frac{a(d-k)p^i}{q-1}\right\rfloor\right)}\overline{\omega}^a((-1)^kd^{-d}\alpha^{-1}k^k(d-k)^{d-k})\\
	&\times\prod\limits_{i=0}^{r-1}\prod\limits_{h=0}^{k-1}\frac{\Gamma_p\left(\left\langle\frac{p^i h}{k}+\frac{p^ia}{q-1}\right\rangle \right)}{\Gamma_p\left(\left\langle\frac{p^i h}{k}\right\rangle \right)}\prod\limits_{h=0}^{d-k-1}\frac{\Gamma_p\left(\left\langle\frac{p^i h}{d-k}+\frac{p^ia}{q-1}\right\rangle \right)}{\Gamma_p\left(\left\langle\frac{p^i h}{d-k}\right\rangle \right)}\prod\limits_{h=0}^{d-1}\frac{\Gamma_p\left(\left\langle\frac{p^i(1+h)}{d}-\frac{p^ia}{q-1}\right\rangle \right)}{\Gamma_p\left(\left\langle\frac{p^i h}{d}\right\rangle \right)}\\
	&=-\frac{1}{q}-\frac{1}{q}\sum_{a=1}^{q-2}(-p)^{\sum_{i=0}^{r-1}\left(-\left\lfloor\frac{-adp^i}{q-1}\right\rfloor-\left\lfloor\frac{akp^i}{q-1}\right\rfloor-\left\lfloor\frac{a(d-k)p^i}{q-1}\right\rfloor\right)}\overline{\omega}^a((-1)^dx)\\
	&\times\prod\limits_{i=0}^{r-1}\prod\limits_{h=0}^{k-1}\frac{\Gamma_p\left(\left\langle\frac{p^i h}{k}+\frac{p^ia}{q-1}\right\rangle \right)}{\Gamma_p\left(\left\langle\frac{p^i h}{k}\right\rangle \right)}\prod\limits_{h=0}^{d-k-1}\frac{\Gamma_p\left(\left\langle\frac{p^i h}{d-k}+\frac{p^ia}{q-1}\right\rangle \right)}{\Gamma_p\left(\left\langle\frac{p^i h}{d-k}\right\rangle \right)}\prod\limits_{h=0}^{d-1}\frac{\Gamma_p\left(\left\langle\frac{p^i(1+h)}{d}-\frac{p^ia}{q-1}\right\rangle \right)}{\Gamma_p\left(\left\langle\frac{p^i h}{d}\right\rangle \right)},
\end{align*}
where $x=\frac{(-1)^{d+k}k^k(d-k)^{d-k}}{\alpha d^d}$. Using Lemma \ref{lemma-3_3}, and Lemma \ref{lemma-3_4} with $l=k$ and $l=d-k$ yield
\begin{align}\label{eqn-1.1}
A&=-\frac{1}{q}-\frac{1}{q}\sum_{a=1}^{q-2}\overline{\omega}^a((-1)^dx)(-p)^{\sum_{i=0}^{r-1} v_{a,i}+1}\prod\limits_{i=0}^{r-1}\Gamma_p\left(\left\langle\frac{p^ia}{q-1}\right\rangle \right)\Gamma_p\left(\left\langle p^i-\frac{p^ia}{q-1}\right\rangle \right)\nonumber\\
&\times\prod\limits_{h=0}^{k-1}\frac{\Gamma_p\left(\left\langle\frac{p^i h}{k}+\frac{p^ia}{q-1}\right\rangle \right)}{\Gamma_p\left(\left\langle\frac{p^i h}{k}\right\rangle \right)}\prod\limits_{h=1}^{d-k-1}\frac{\Gamma_p\left(\left\langle\frac{p^i h}{d-k}+\frac{p^ia}{q-1}\right\rangle \right)}{\Gamma_p\left(\left\langle\frac{p^i h}{d-k}\right\rangle \right)}\prod\limits_{h=1}^{d-1}\frac{\Gamma_p\left(\left\langle\frac{p^i h}{d}-\frac{p^ia}{q-1}\right\rangle \right)}{\Gamma_p\left(\left\langle\frac{p^i h}{d}\right\rangle \right)},
\end{align}
where $$v_{a,i}=-\sum_{h=1}^{d-1}\left\lfloor\left\langle\frac{hp^i}{d}\right\rangle-\frac{ap^i}{q-1}\right\rfloor-\sum_{h=0}^{k-1}\left\lfloor\left\langle\frac{-hp^i}{k}\right\rangle+\frac{ap^i}{q-1}\right\rfloor-\sum_{h=1}^{d-k-1}\left\lfloor\left\langle\frac{-hp^i}{d-k}\right\rangle+\frac{ap^i}{q-1}\right\rfloor.$$
For a positive integer $l$, one can easily check that the following equality holds
\begin{align}\label{eqn-1.2}
\prod\limits_{h=0}^{l-1}\frac{\Gamma_p\left(\left\langle\frac{p^i h}{l}+\frac{p^ia}{q-1}\right\rangle \right)}{\Gamma_p\left(\left\langle\frac{p^i h}{l}\right\rangle \right)}=\prod\limits_{h=0}^{l-1}\frac{\Gamma_p\left(\left\langle\frac{-p^i h}{l}+\frac{p^ia}{q-1}\right\rangle \right)}{\Gamma_p\left(\left\langle\frac{-p^i h}{l}\right\rangle \right)}.
\end{align}
Substituting \eqref{eqn-1.2} with $l=k$ and $l=d-k$ in \eqref{eqn-1.1} and using \eqref{eq-12}, we deduce that  
\begin{align*}
A&=-\frac{1}{q}-\frac{(-p)^r}{q}\sum_{a=1}^{q-2}\overline{\omega}^a((-1)^d(-x))(-p)^{\sum_{i=0}^{r-1} v_{a,i}}(-1)^r\prod\limits_{i=0}^{r-1}\prod\limits_{h=0}^{k-1}\frac{\Gamma_p\left(\left\langle\frac{-p^i h}{k}+\frac{p^ia}{q-1}\right\rangle \right)}{\Gamma_p\left(\left\langle\frac{-p^i h}{k}\right\rangle \right)} \\
&\hspace{.4cm}\times\prod\limits_{h=1}^{d-k-1}\frac{\Gamma_p\left(\left\langle\frac{-p^i h}{d-k}+\frac{p^ia}{q-1}\right\rangle \right)}{\Gamma_p\left(\left\langle\frac{-p^i h}{d-k}\right\rangle \right)}\prod\limits_{h=1}^{d-1}\frac{\Gamma_p\left(\left\langle\frac{p^i h}{d}-\frac{p^ia}{q-1}\right\rangle \right)}{\Gamma_p\left(\left\langle\frac{p^i h}{d}\right\rangle \right)}.
\end{align*}
Adding and subtracting the term under summation for $a=0$ and using $\overline{\omega}^a(-1)=(-1)^a$, we obtain
\begin{align*}
A&=\frac{q-1}{q}-\sum_{a=0}^{q-2}(-1)^{a(d-1)}\overline{\omega}^a(x)(-p)^{\sum_{i=0}^{r-1} v_{a,i}}\prod\limits_{i=0}^{r-1}\prod\limits_{h=0}^{k-1}\frac{\Gamma_p\left(\left\langle\frac{-p^i h}{k}+\frac{p^ia}{q-1}\right\rangle \right)}{\Gamma_p\left(\left\langle\frac{-p^i h}{k}\right\rangle \right)} \\
&\hspace{.4cm}\times\prod\limits_{h=1}^{d-k-1}\frac{\Gamma_p\left(\left\langle\frac{-p^i h}{d-k}+\frac{p^ia}{q-1}\right\rangle \right)}{\Gamma_p\left(\left\langle\frac{-p^i h}{d-k}\right\rangle \right)}\prod\limits_{h=1}^{d-1}\frac{\Gamma_p\left(\left\langle\frac{p^i h}{d}-\frac{p^ia}{q-1}\right\rangle \right)}{\Gamma_p\left(\left\langle\frac{p^i h}{d}\right\rangle \right)}\\
&=\frac{q-1}{q}+(q-1)\cdot{_{d-1}}G_{d-1}\left[\begin{array}{ccccccc}
			\frac{1}{d}, & \frac{2}{d}, & \ldots, & \frac{k}{d}, & \frac{k+1}{d}, & \ldots , & \frac{d-1}{d} \vspace{0.1cm}\\
		   	0,& \frac{1}{k}, & \ldots, & \frac{k-1}{k}, & \frac{1}{d-k},  & \ldots , & \frac{d-k-1}{d-k} 
		\end{array}|x
		\right]_q.
\end{align*}
Substituting the expression for $A$ in \eqref{neweqn-01} and then using Lemma \ref{lemma2_1} yield
\begin{align}\label{eqn-1.4}
C(d,k,\alpha)&=\frac{q-1}{q}+(1-q)\sum_{\begin{subarray}{1} \ \chi\in\widehat{\mathbb{F}_q^{\times}}, \\ \chi^{k_1}= \varepsilon \end{subarray}}\chi((-1)^d\alpha)-\frac{1-q}{q}\sum_{\begin{subarray}{1} \ \chi\in\widehat{\mathbb{F}_q^{\times}}, \\ \chi^{k_1}= \varepsilon \end{subarray}}(q-1)\chi(\alpha)\delta(\chi^d) \notag\\
&+(q-1)\cdot{_{d-1}}G_{d-1}\left[\begin{array}{ccccccc}
			\frac{1}{d}, & \frac{2}{d}, & \ldots, & \frac{k}{d}, & \frac{k+1}{d}, & \ldots , & \frac{d-1}{d} \vspace{0.1cm}\\
			0,& \frac{1}{k}, & \ldots, & \frac{k-1}{k}, & \frac{1}{d-k},  & \ldots , & \frac{d-k-1}{d-k}  
		\end{array}|x
		\right]_q.
\end{align}
From \eqref{eqn-1.5}, we have 
\begin{align}\label{neweqn-1.5}
	C(d,k,\alpha)=(q-1)n_q(\alpha)-(q-1)\sum_{\begin{subarray}{1} \ \chi\in\widehat{\mathbb{F}_q^{\times}}, \\ \chi^{k_1}= \varepsilon \end{subarray}}\chi((-1)^d\alpha),
\end{align}
where $n_q(\alpha)$ is the number of distinct zeros of the polynomial $h_\alpha(y)$ in $\mathbb{F}_q$. Taking $\alpha=(-1)^{d-k}(\lambda d)^{-d}$, we have $n_q(\alpha)=r_q(\lambda)$. Combining \eqref{eqn-1.4} and \eqref{neweqn-1.5}, and taking $\alpha=(-1)^{d-k}(\lambda d)^{-d}$, we obtain the required expression for $r_q(\lambda)$.
\end{proof}
\begin{remark}
If we combine \eqref{eqn-1.4} and \eqref{neweqn-1.5}, and take $d=3n$, $k=2$, and $r=1$, then we obtain \cite[Theorem 1.2]{NS1}.
\end{remark}
\section{Proof of Theorems \ref{MT-3} and \ref{MT-4}}
In this section, we prove the summation identites given in Theorems \ref{MT-3} and \ref{MT-4}. To prove these identities, we need two lemmas which connect the product of certain values of the $p$-adic gamma function to some character sums. 
\begin{lemma}\label{lemma4.1}
	For $0 \leq a\leq q-2$, we have
	\begin{align*}
		\prod\limits_{i=0}^{r-1} \frac{ \Gamma_p\left(\left\langle p^i-\frac{ap^i}{q-1}\right\rangle\right) \Gamma_p\left(\left\langle \frac{p^i}{2}+\frac{ap^i}{q-1}\right\rangle\right)}{ (-p)^{\left\lfloor \frac{1}{2}+\frac{ap^i}{q-1}\right\rfloor}\Gamma_p\left(\left\langle \frac{p^i}{2}\right\rangle\right)}=-\sum_{t\in\mathbb{F}_q}\varphi(t(t-1))\overline{\omega}^a(-t) \prod\limits_{i=0}^{r-1} (-p)^{\left\lfloor \frac{-ap^i}{q-1}\right\rfloor}.
	\end{align*}
\end{lemma}
\begin{proof}
We have 
	\begin{align*}
		&	\prod\limits_{i=0}^{r-1} \frac{ \Gamma_p\left(\left\langle p^i-\frac{ap^i}{q-1}\right\rangle\right) \Gamma_p\left(\left\langle \frac{p^i}{2}+\frac{ap^i}{q-1}\right\rangle\right)}{ (-p)^{\left\lfloor \frac{1}{2}+\frac{ap^i}{q-1}\right\rfloor}\Gamma_p\left(\left\langle \frac{p^i}{2}\right\rangle\right)}\\
		&\hspace*{1cm}=\prod\limits_{i=0}^{r-1} \frac{(-p)^{\left\langle \frac{1}{2}+\frac{ap^i}{q-1}\right\rangle} \Gamma_p\left(\left\langle p^i-\frac{ap^i}{q-1}\right\rangle\right) \Gamma_p\left(\left\langle \frac{p^i}{2}+\frac{ap^i}{q-1}\right\rangle\right)}{(-p)^{\frac{1}{2}+\frac{ap^i}{q-1}} \Gamma_p\left(\left\langle \frac{p^i}{2}\right\rangle\right)}\\
			&\hspace*{1cm}=\prod\limits_{i=0}^{r-1} \frac{(-p)^{\left\langle \frac{p^i}{2}+\frac{ap^i}{q-1}\right\rangle} \Gamma_p\left(\left\langle p^i-\frac{ap^i}{q-1}\right\rangle\right) \Gamma_p\left(\left\langle \frac{p^i}{2}+\frac{ap^i}{q-1}\right\rangle\right)}{(-p)^{\left\langle\frac{p^i}{2}\right\rangle}(-p)^{\frac{ap^i}{q-1}} \Gamma_p\left(\left\langle \frac{p^i}{2}\right\rangle\right)}\\
			&\hspace*{1cm}=\prod\limits_{i=0}^{r-1} \frac{(-p)^{\left\langle \frac{p^i}{2}+\frac{ap^i}{q-1}\right\rangle}(-p)^{\left\langle-\frac{ap^i}{q-1}\right\rangle} \Gamma_p\left(\left\langle p^i-\frac{ap^i}{q-1}\right\rangle\right) \Gamma_p\left(\left\langle \frac{p^i}{2}+\frac{ap^i}{q-1}\right\rangle\right)}{(-p)^{\left\langle\frac{p^i}{2}\right\rangle}(-p)^{\frac{ap^i}{q-1}+\left\langle-\frac{ap^i}{q-1}\right\rangle} \Gamma_p\left(\left\langle \frac{p^i}{2}\right\rangle\right)}.
	\end{align*}
Using Gross-Koblitz formula, we obtain
\begin{align}\label{eq-22}
	\prod\limits_{i=0}^{r-1} \frac{ \Gamma_p\left(\left\langle p^i-\frac{ap^i}{q-1}\right\rangle\right) \Gamma_p\left(\left\langle \frac{p^i}{2}+\frac{ap^i}{q-1}\right\rangle\right)}{ (-p)^{\left\lfloor \frac{1}{2}+\frac{ap^i}{q-1}\right\rfloor}\Gamma_p\left(\left\langle \frac{p^i}{2}\right\rangle\right)}=\frac{-g(\varphi \overline{\omega}^a) g(\omega^a)}{g(\varphi)}\prod_{i=0}^{r-1}(-p)^{\left\lfloor \frac{-ap^i}{q-1}\right\rfloor} .
\end{align}
Lemma \ref{lemma2_2} and \eqref{eq-1} yield
\begin{align}\label{eq-23}
	\frac{-g(\varphi \overline{\omega}^a) g(\omega^a)}{g(\varphi)}&=-J(\varphi\overline{\omega}^a, \omega^a)\nonumber\\
	&=-\varphi\overline{\omega}^a(-1) J(\varphi\overline{\omega}^a,\varphi)\nonumber\\
	&=-\varphi\overline{\omega}^a(-1)\sum_{t\in\mathbb{F}_q}\varphi(1-t)\varphi\overline{\omega}^a(t)\nonumber\\
	&=-\sum_{t\in\mathbb{F}_q}\varphi(t(t-1))\overline{\omega}^a(-t).
\end{align}
Combining \eqref{eq-22} and \eqref{eq-23}, we obtain the required result.
\end{proof}
\begin{lemma}\label{lemma4.2}
	For $0\leq a\leq q-2$, we have
	\begin{align*}
		\prod\limits_{i=0}^{r-1} \frac{ \Gamma_p\left(\left\langle \frac{ap^i}{q-1}\right\rangle\right) \Gamma_p\left(\left\langle \frac{p^i}{2}-\frac{ap^i}{q-1}\right\rangle\right)}{ (-p)^{\left\lfloor \frac{1}{2}-\frac{ap^i}{q-1}\right\rfloor}\Gamma_p\left(\left\langle \frac{p^i}{2}\right\rangle\right)}=-\sum_{t\in\mathbb{F}_q}\varphi(t(t-1))\omega^a(-t) \prod\limits_{i=0}^{r-1} (-p)^{\left\lfloor \frac{ap^i}{q-1}\right\rfloor}.
	\end{align*}
\end{lemma}
\begin{proof}
	The proof is similar to that of Lemma \ref{lemma4.1}.
\end{proof}
Having Lemma \ref{lemma4.1} and Lemma \ref{lemma4.2} showed, we are ready to prove Theorem \ref{MT-3}.
\begin{proof}[Proof of Theorem \ref{MT-3}]
	For $x\in\mathbb{F}_q^\times$, we consider
	\begin{align}\label{eq-25}
		A_x&:= q\cdot{_{d-1}}G_{d-1}\left[\begin{array}{ccccccc}
		\frac{1}{d}, &\frac{2}{d}, & \ldots, & \frac{k}{d}, & \frac{k+1}{d}, & \ldots , & \frac{d-1}{d} \\
		0, &\frac{1}{k}, & \ldots, & \frac{k-1}{k}, & \frac{1}{d-k}, & \ldots , & \frac{d-k-1}{d-k} 
	\end{array}|x
	\right]_q\notag\\
	&=-\frac{q}{q-1} -\frac{q}{q-1}\sum_{a=1}^{q-2}\overline{\omega}^{a}(x)(-p)^{\sum\limits_{i=0}^{r-1} v_{a,i}}\prod_{i=0}^{r-1}\prod\limits_{h=1}^{d-1}\frac{\Gamma_p\left(\left\langle\frac{p^ih}{d}-\frac{p^ia}{q-1}\right\rangle \right)}{\Gamma_p\left(\left\langle \frac{hp^i}{d}\right\rangle\right)}\notag\\
	&\hspace*{.5cm}\times\prod\limits_{h=0}^{k-1}\frac{\Gamma_p\left(\left\langle\frac{-p^i h}{k}+\frac{p^ia}{q-1}\right\rangle \right)}{\Gamma_p\left(\left\langle \frac{-hp^i}{k}\right\rangle\right)}\prod\limits_{h=1}^{d-k-1}\frac{\Gamma_p\left(\left\langle\frac{-p^i h}{d-k}+\frac{p^ia}{q-1}\right\rangle \right)}{\Gamma_p\left(\left\langle \frac{-hp^i}{d-k}\right\rangle\right)},
\end{align}
where $v_{a,i}=-\sum\limits_{h=1}^{d-1} \left\lfloor\left\langle\frac{hp^i}{d}\right\rangle -\frac{ap^i}{q-1}\right\rfloor-\sum\limits_{h=0}^{k-1} \left\lfloor\left\langle\frac{-hp^i}{k}\right\rangle+\frac{ap^i}{q-1}\right\rfloor-\sum\limits_{h=1}^{d-k-1} \left\lfloor\left\langle\frac{-hp^i}{d-k}\right\rangle+\frac{ap^i}{q-1}\right\rfloor.$
Since $d$ and $k$ are odd integers, so $d-k$ is even. Now, using \eqref{eq-12} and the fact that $\frac{d-k}{2}\in \mathbb{Z}$, we rewrite \eqref{eq-25} as
\begin{align*}
		A_x&= -\frac{(-1)^r q}{q-1}\sum_{a=1}^{q-2}\overline{\omega}^{a}(-x)(-p)^{\sum\limits_{i=0}^{r-1} T_{a,i}}\prod_{i=0}^{r-1}\frac{\Gamma_p\left(\left\langle p^i-\frac{p^ia}{q-1}\right\rangle \right)\Gamma_p\left(\left\langle\frac{p^i}{2}+\frac{p^ia}{q-1}\right\rangle \right)}{(-p)^{\left\lfloor \frac{1}{2}+\frac{ap^i}{q-1}\right\rfloor}\Gamma_p\left(\left\langle\frac{p^i}{2}\right\rangle \right)}\\
	&\hspace*{.5cm}\times\Gamma_p^2\left(\left\langle\frac{p^ia}{q-1}\right\rangle \right)\prod\limits_{h=1}^{d-1}\frac{\Gamma_p\left(\left\langle\frac{p^ih}{d}-\frac{p^ia}{q-1}\right\rangle \right)}{\Gamma_p\left(\left\langle \frac{hp^i}{d}\right\rangle\right)}\prod\limits_{h=1}^{k-1}\frac{\Gamma_p\left(\left\langle\frac{-p^i h}{k}+\frac{p^ia}{q-1}\right\rangle \right)}{\Gamma_p\left(\left\langle \frac{-hp^i}{k}\right\rangle\right)}\\
	&\hspace*{.5cm}\times\prod\limits_{\begin{subarray}{1} \ h=1 \\ h\neq\frac{d-k}{2}\end{subarray}}^{d-k-1}\frac{\Gamma_p\left(\left\langle\frac{-p^i h}{d-k}+\frac{p^ia}{q-1}\right\rangle \right)}{\Gamma_p\left(\left\langle \frac{-hp^i}{d-k}\right\rangle\right)}-\frac{q}{q-1},
\end{align*}
where $T_{a,i}=-\sum\limits_{h=1}^{d-1} \left\lfloor\left\langle\frac{hp^i}{d}\right\rangle -\frac{ap^i}{q-1}\right\rfloor-\sum\limits_{h=0}^{k-1} \left\lfloor\left\langle\frac{-hp^i}{k}\right\rangle+\frac{ap^i}{q-1}\right\rfloor-\sum\limits_{\begin{subarray}{1} \ h=0 \\ h\neq\frac{d-k}{2}\end{subarray}}^{d-k-1} \left\lfloor\left\langle\frac{-hp^i}{d-k}\right\rangle+\frac{ap^i}{q-1}\right\rfloor+\left\lfloor \frac{ap^i}{q-1}\right\rfloor.$ Lemma \ref{lemma4.1} yields
\begin{align*}
		A_x&=-\frac{q}{q-1} +\frac{(-1)^rq}{q-1}\sum_{a=1}^{q-2}\overline{\omega}^{a}(-x)(-p)^{\sum\limits_{i=0}^{r-1} T_{a,i}}\sum_{t\in\mathbb{F}_q}\varphi(t(t-1))\overline{\omega}^a(-t) \prod_{i=0}^{r-1}(-p)^{\left\lfloor \frac{-ap^i}{q-1}\right\rfloor}\\
	&\hspace*{.5cm}\times\Gamma_p^2\left(\left\langle\frac{p^ia}{q-1}\right\rangle \right)\prod\limits_{h=1}^{d-1}\frac{\Gamma_p\left(\left\langle\frac{p^ih}{d}-\frac{p^ia}{q-1}\right\rangle \right)}{\Gamma_p\left(\left\langle \frac{hp^i}{d}\right\rangle\right)}\prod\limits_{h=1}^{k-1}\frac{\Gamma_p\left(\left\langle\frac{-p^i h}{k}+\frac{p^ia}{q-1}\right\rangle \right)}{\Gamma_p\left(\left\langle \frac{-hp^i}{k}\right\rangle\right)}\\
	&\hspace*{.5cm}\times\prod\limits_{\begin{subarray}{1} \ h=1 \\ h\neq\frac{d-k}{2}\end{subarray}}^{d-k-1}\frac{\Gamma_p\left(\left\langle\frac{-p^i h}{d-k}+\frac{p^ia}{q-1}\right\rangle \right)}{\Gamma_p\left(\left\langle \frac{-hp^i}{d-k}\right\rangle\right)}\\
	&=-\frac{q}{q-1} +\frac{(-1)^rq}{q-1}\sum_{a=1}^{q-2}\sum_{t\in\mathbb{F}_q}\varphi(t(t-1))\overline{\omega}^{a}(xt)(-p)^{\sum\limits_{i=0}^{r-1} T_{a,i}+\left\lfloor \frac{-ap^i}{q-1}\right\rfloor}\prod_{i=0}^{r-1}\Gamma_p^2\left(\left\langle\frac{p^ia}{q-1}\right\rangle \right)\\
	&\hspace*{.5cm}\times \prod\limits_{h=1}^{d-1}\frac{\Gamma_p\left(\left\langle\frac{p^ih}{d}-\frac{p^ia}{q-1}\right\rangle \right)}{\Gamma_p\left(\left\langle \frac{hp^i}{d}\right\rangle\right)}\prod\limits_{h=1}^{k-1}\frac{\Gamma_p\left(\left\langle\frac{-p^i h}{k}+\frac{p^ia}{q-1}\right\rangle \right)}{\Gamma_p\left(\left\langle \frac{-hp^i}{k}\right\rangle\right)}\prod\limits_{\begin{subarray}{1} \ h=1 \\ h\neq\frac{d-k}{2}\end{subarray}}^{d-k-1}\frac{\Gamma_p\left(\left\langle\frac{-p^i h}{d-k}+\frac{p^ia}{q-1}\right\rangle \right)}{\Gamma_p\left(\left\langle \frac{-hp^i}{d-k}\right\rangle\right)}.
\end{align*}
Since $1\leq a\leq q-2$ and $\gcd(p^i,q-1)=1$, therefore $\frac{ap^i}{q-1}$ is not an integer. Also, $\lfloor x\rfloor+\lfloor-x\rfloor =-1$ if $x\not\in\mathbb{Z}$, and hence $\lfloor\frac{ap^i}{q-1}\rfloor + \lfloor\frac{-ap^i}{q-1}\rfloor=-1$. Thus, 
\begin{align*}
	A_x&=-\frac{q}{q-1} +\frac{1}{q-1}\sum_{a=1}^{q-2}\sum_{t\in\mathbb{F}_q}\varphi(t(t-1))\overline{\omega}^{a}(xt)(-p)^{\sum\limits_{i=0}^{r-1} u_{a,i}}\prod_{i=0}^{r-1}\Gamma_p^2\left(\left\langle\frac{p^ia}{q-1}\right\rangle \right)\\
	&\hspace*{.5cm}\times \prod\limits_{h=1}^{d-1}\frac{\Gamma_p\left(\left\langle\frac{p^ih}{d}-\frac{p^ia}{q-1}\right\rangle \right)}{\Gamma_p\left(\left\langle \frac{hp^i}{d}\right\rangle\right)}\prod\limits_{h=1}^{k-1}\frac{\Gamma_p\left(\left\langle\frac{-p^i h}{k}+\frac{p^ia}{q-1}\right\rangle \right)}{\Gamma_p\left(\left\langle \frac{-hp^i}{k}\right\rangle\right)}\prod\limits_{\begin{subarray}{1} \ h=1 \\ h\neq\frac{d-k}{2}\end{subarray}}^{d-k-1}\frac{\Gamma_p\left(\left\langle\frac{-p^i h}{d-k}+\frac{p^ia}{q-1}\right\rangle \right)}{\Gamma_p\left(\left\langle \frac{-hp^i}{d-k}\right\rangle\right)},
\end{align*}
where $u_{a,i}=-\sum\limits_{h=1}^{d-1} \left\lfloor\left\langle\frac{hp^i}{d}\right\rangle -\frac{ap^i}{q-1}\right\rfloor-\sum\limits_{h=0}^{k-1} \left\lfloor\left\langle\frac{-hp^i}{k}\right\rangle+\frac{ap^i}{q-1}\right\rfloor-\sum\limits_{\begin{subarray}{1} \ h=0 \\ h\neq\frac{d-k}{2}\end{subarray}}^{d-k-1} \left\lfloor\left\langle\frac{-hp^i}{d-k}\right\rangle+\frac{ap^i}{q-1}\right\rfloor.$ Adding and subtracting the term under the summation for $a=0$ gives
\begin{align}
&   A_x=-\frac{q}{q-1}-\frac{1}{q-1}\sum_{t\in\mathbb{F}_q}\varphi(t(t-1))+	\frac{1}{q-1}\sum_{a=0}^{q-2}\sum_{t\in\mathbb{F}_q}\varphi(t(t-1))\overline{\omega}^{a}(xt)(-p)^{\sum\limits_{i=0}^{r-1} u_{a,i}}\nonumber\\
   &\hspace*{.5cm}\times\prod_{i=0}^{r-1} \prod\limits_{h=1}^{d-1}\frac{\Gamma_p\left(\left\langle\frac{p^ih}{d}-\frac{p^ia}{q-1}\right\rangle \right)}{\Gamma_p\left(\left\langle \frac{hp^i}{d}\right\rangle\right)}\prod\limits_{h=0}^{k-1}\frac{\Gamma_p\left(\left\langle\frac{-p^i h}{k}+\frac{p^ia}{q-1}\right\rangle \right)}{\Gamma_p\left(\left\langle \frac{-hp^i}{k}\right\rangle\right)}\prod\limits_{\begin{subarray}{1} \ h=0 \\ h\neq\frac{d-k}{2}\end{subarray}}^{d-k-1}\frac{\Gamma_p\left(\left\langle\frac{-p^i h}{d-k}+\frac{p^ia}{q-1}\right\rangle \right)}{\Gamma_p\left(\left\langle \frac{-hp^i}{d-k}\right\rangle\right)}\label{eq-27}\\
   &=-1-\sum_{t\in\mathbb{F}_q}\varphi(t(t-1)) \times\nonumber\\
   &{_{d-1}}G_{d-1}\left[\begin{array}{cccccccccc}
   	\frac{1}{d}, \hspace{-.2cm}& \ldots, \hspace{-.2cm}& \frac{k}{d}, \hspace{-.2cm}& \frac{k+1}{d}, \hspace{-.2cm}& \ldots, \hspace{-.2cm}& \frac{k+\frac{d-k}{2}-1}{d}, \hspace{-.2cm}& \frac{k+\frac{d-k}{2}}{d}, \hspace{-.2cm}& \ldots,\hspace{-.2cm} & \frac{d-2}{d}, \hspace{-.2cm}& \frac{d-1}{d}\vspace{0.1cm} \\
   	0, \hspace{-.2cm}& \ldots, \hspace{-.2cm}& \frac{k-1}{k}, \hspace{-.2cm}& \frac{1}{d-k}, \hspace{-.2cm}& \ldots,\hspace{-.2cm}& \frac{\frac{d-k}{2}-1}{d-k},\hspace{-.2cm}& \frac{\frac{d-k}{2}+1}{d-k},\hspace{-.2cm}& \ldots, \hspace{-.2cm}& \frac{d-k-1}{d-k}, \hspace{-.2cm}& 0 
   \end{array}|xt
   \right]_q.\notag
\end{align}
 Using \eqref{eq-0} and \eqref{eq-4},
we obtain $\sum_{t\in\mathbb{F}_q}\varphi(t(t-1))=-1$ and substituting this value in the first summation of \eqref{eq-27} gives the last equality. This completes the proof of the theorem.
\end{proof}
\begin{proof}[Proof of Theorem \ref{MT-4}]
	For $x\in\mathbb{F}_q^\times$, we consider
\begin{align}\label{eq-28}
	B_x&:= {_{d-1}}G_{d-1}\left[\begin{array}{cccccc}
		\frac{1}{d}, & \ldots, & \frac{k}{d}, & \frac{k+1}{d}, & \ldots , & \frac{d-1}{d}\vspace{0.1cm} \\
		0, & \ldots, & \frac{k-1}{k}, & \frac{1}{d-k}, & \ldots , & \frac{d-k-1}{d-k} 
	\end{array}|x
	\right]_q\notag\\
	&=-\frac{1}{q-1}\sum_{a=0}^{q-2}\overline{\omega}^{a}(-x)(-p)^{\sum_{i=0}^{r-1}N_{a,i}}\prod_{i=0}^{r-1}\Gamma_p\left(\left\langle\frac{p^ia}{q-1}\right\rangle \right)\nonumber\\
&\hspace*{.5cm}\times\prod\limits_{h=1}^{d-1}\frac{\Gamma_p\left(\left\langle\frac{p^ih}{d}-\frac{p^ia}{q-1}\right\rangle \right)}{\Gamma_p\left(\left\langle \frac{hp^i}{d}\right\rangle\right)}\prod\limits_{h=1}^{d-k-1}\frac{\Gamma_p\left(\left\langle\frac{-p^i h}{d-k
	}+\frac{p^ia}{q-1}\right\rangle \right)}{\Gamma_p\left(\left\langle \frac{-hp^i}{d-k}\right\rangle\right)}\prod\limits_{h=1}^{k-1}\frac{\Gamma_p\left(\left\langle\frac{-p^i h}{k
	}+\frac{p^ia}{q-1}\right\rangle \right)}{\Gamma_p\left(\left\langle \frac{-hp^i}{k}\right\rangle\right)}
\end{align}
where $N_{a,i}=-\sum\limits_{h=1}^{d-1} \left\lfloor\left\langle\frac{hp^i}{d}\right\rangle -\frac{ap^i}{q-1}\right\rfloor- \sum\limits_{h=0}^{k-1} \left\lfloor\left\langle\frac{-hp^i}{k}\right\rangle+\frac{ap^i}{q-1}\right\rfloor- \sum\limits_{h=1}^{d-k-1} \left\lfloor\left\langle\frac{-hp^i}{d-k}\right\rangle+\frac{ap^i}{q-1}\right\rfloor.$
Since $d$ is an even integers, $\frac{d}{2}\in \mathbb{Z}$. We rewrite \eqref{eq-28} as
\begin{align*}
	B_x&=-\frac{1}{q-1}\sum_{a=0}^{q-2}\overline{\omega}^{a}(-x)(-p)^{\sum\limits_{i=0}^{r-1}N_{a,i}+\left\lfloor\left\langle\frac{p^i}{2}\right\rangle -\frac{ap^i}{q-1}\right\rfloor}\prod_{i=0}^{r-1}\frac{\Gamma_p\left(\left\langle\frac{p^ia}{q-1}\right\rangle \right)}{(-p)^{\left\lfloor\left\langle\frac{p^i}{2}\right\rangle -\frac{ap^i}{q-1}\right\rfloor}}\frac{\Gamma_p\left(\left\langle\frac{p^i}{2}-\frac{p^ia}{q-1}\right\rangle \right)}{\Gamma_p\left(\left\langle\frac{p^i}{2}\right\rangle \right)}\nonumber\\
&\hspace*{.5cm}\times\prod\limits_{\begin{subarray}{1} \ h=1 \\ h\neq\frac{d}{2}\end{subarray}}^{d-1}\frac{\Gamma_p\left(\left\langle\frac{p^ih}{d}-\frac{p^ia}{q-1}\right\rangle \right)}{\Gamma_p\left(\left\langle \frac{hp^i}{d}\right\rangle\right)}\prod\limits_{h=1}^{d-k-1}\frac{\Gamma_p\left(\left\langle\frac{-p^i h}{d-k
	}+\frac{p^ia}{q-1}\right\rangle \right)}{\Gamma_p\left(\left\langle \frac{-hp^i}{d-k}\right\rangle\right)}\prod\limits_{h=1}^{k-1}\frac{\Gamma_p\left(\left\langle\frac{-p^i h}{k
	}+\frac{p^ia}{q-1}\right\rangle \right)}{\Gamma_p\left(\left\langle \frac{-hp^i}{k}\right\rangle\right)}.
\end{align*}
Using Lemma \ref{lemma4.2}, we obtain
\begin{align*}
	B_x&=\frac{1}{q-1}\sum_{a=0}^{q-2}\sum_{t\in\mathbb{F}_q^\times}\varphi(t(t-1))\overline{\omega}^a\left(\frac{x}{t}\right) (-p)^{\sum\limits_{i=0}^{r-1} v_{i,a}}\\
	&\hspace*{.5cm}\times\prod\limits_{\begin{subarray}{1} \ h=1 \\ h\neq\frac{d}{2}\end{subarray}}^{d-1}\frac{\Gamma_p\left(\left\langle\frac{p^ih}{d}-\frac{p^ia}{q-1}\right\rangle \right)}{\Gamma_p\left(\left\langle \frac{hp^i}{d}\right\rangle\right)}\prod\limits_{h=1}^{d-k-1}\frac{\Gamma_p\left(\left\langle\frac{-p^i h}{d-k
		}+\frac{p^ia}{q-1}\right\rangle \right)}{\Gamma_p\left(\left\langle \frac{-hp^i}{d-k}\right\rangle\right)}\prod\limits_{h=1}^{k-1}\frac{\Gamma_p\left(\left\langle\frac{-p^i h}{k
		}+\frac{p^ia}{q-1}\right\rangle \right)}{\Gamma_p\left(\left\langle \frac{-hp^i}{k}\right\rangle\right)},
\end{align*}
where $v_{i,a}=-\sum\limits_{\begin{subarray}{1} \ h=1 \\ h\neq\frac{d}{2}\end{subarray}}^{d-1} \left\lfloor\left\langle\frac{hp^i}{d}\right\rangle -\frac{ap^i}{q-1}\right\rfloor- \sum\limits_{h=1}^{k-1} \left\lfloor\left\langle\frac{-hp^i}{k}\right\rangle+\frac{ap^i}{q-1}\right\rfloor- \sum\limits_{h=1}^{d-k-1} \left\lfloor\left\langle\frac{-hp^i}{d-k}\right\rangle+\frac{ap^i}{q-1}\right\rfloor.$ Thus, we have 
\begin{align*}
	B_x&=-\sum\limits_{t\in\mathbb{F}_q^\times}\varphi(t(t-1)) {_{d-2}}G_{d-2}\left[\begin{array}{cccccc}
		\frac{1}{d}, & \ldots, & \frac{\frac{d}{2}-1}{d}, & \frac{\frac{d}{2}+1}{d},& \ldots,  & \frac{d-1}{d} \\
	c_1, & \ldots, & c_{\frac{d}{2}-1}, & c_{\frac{d}{2}}, & \ldots, & c_{d-2} 
	\end{array}|\frac{x}{t}
	\right]_q,
	\end{align*}
where $c$'s are as defined in the statement of the theorem. Now, using the translation $t \mapsto t^{-1}$, we obtain the required identity.
\end{proof} 
\section{Applications to elliptic curves}
In this section, we prove Theorems \ref{MT-5}, \ref{MT-7}, and \ref{MT-6} which provide certain identities for the trace of Frobenius of elliptic curves. Firstly, we prove a lemma that gives a relation between $_3G_3[\cdots]$ and $_2G_2[\cdots]$, and this lemma will be used to prove our main results.
\begin{lemma}\label{lemma_4.3}
Let $p$ be an odd prime and $q=p^r$, $r\geq1$. For $x\in\mathbb{F}_q$, we have
\begin{align*}
	{_{3}}G_{3}\left[\begin{array}{ccc}
		\frac{1}{4},  &\frac{1}{2}, & \frac{3}{4} \vspace*{0.1cm}\\
		0, & \frac{1}{2}, & \frac{1}{2}
	\end{array}|x
	\right]_q=	{_{2}}G_{2}\left[\begin{array}{cc}
		\frac{1}{4},  & \frac{3}{4} \vspace*{0.1cm}\\
		0, & \frac{1}{2}
	\end{array}|x
	\right]_q +\frac{\varphi(x)}{q}.
\end{align*}
\end{lemma}
\begin{proof}
For $x\in\mathbb{F}_q$, we have
	\begin{align*}
	&{_{3}}G_{3}\left[\begin{array}{ccc}
		\frac{1}{4},  &\frac{1}{2}, & \frac{3}{4}\vspace*{0.1cm} \\
		0, & \frac{1}{2}, & \frac{1}{2}
	\end{array}|x
	\right]_q\\
	&=-\frac{1}{q-1}\sum_{a=0}^{q-2}(-1)^{3a}\overline{\omega}^a(x)(-p)^{\sum_{i=0}^{r-1}\beta_{i,a}}\prod\limits_{i=0}^{r-1}\Gamma_{p}\left(\left\langle\frac{ap^i}{q-1}\right\rangle\right)M_{i,a}\\
	&\times\frac{\Gamma_{p}\left(\left\langle\frac{p^i}{4}-\frac{ap^i}{q-1}\right\rangle\right)\Gamma_{p}\left(\left\langle\frac{3p^i}{4}-\frac{ap^i}{q-1}\right\rangle\right)\Gamma_{p}\left(\left\langle\frac{-p^i}{2}+\frac{ap^i}{q-1}\right\rangle\right)}{\Gamma_p\left(\left\langle\frac{p^i}{4}\right\rangle\right)\Gamma_p\left(\left\langle\frac{3p^i}{4}\right\rangle\right)\Gamma_p\left(\left\langle\frac{-p^i}{2}\right\rangle\right)},
\end{align*}
where $\beta_{i,a}=-\left\lfloor\left\langle\frac{p^i}{4}\right\rangle-\frac{ap^i}{q-1}\right\rfloor-\left\lfloor\left\langle\frac{3p^i}{4}\right\rangle-\frac{ap^i}{q-1}\right\rfloor-\left\lfloor\left\langle\frac{p^i}{2}\right\rangle-\frac{ap^i}{q-1}\right\rfloor-\left\lfloor\left\langle\frac{-p^i}{2}\right\rangle+\frac{ap^i}{q-1}\right\rfloor-\left\lfloor\left\langle\frac{-p^i}{2}\right\rangle+\frac{ap^i}{q-1}\right\rfloor-\left\lfloor\frac{ap^i}{q-1}\right\rfloor$ and $M_{i,a}=\frac{\Gamma_{p}\left(\left\langle\frac{p^i}{2}-\frac{ap^i}{q-1}\right\rangle\right)\Gamma_{p}\left(\left\langle\frac{-p^i}{2}+\frac{ap^i}{q-1}\right\rangle\right)}{\Gamma_p\left(\left\langle\frac{p^i}{2}\right\rangle\right)\Gamma_p\left(\left\langle\frac{-p^i}{2}\right\rangle\right)}$. Taking out the term for $a=\frac{q-1}{2}$ and then using \eqref{eq-13} and the fact that $\left\lfloor\left\langle\frac{p^i}{2}\right\rangle-\frac{ap^i}{q-1}\right\rfloor+\left\lfloor\left\langle\frac{-p^i}{2}\right\rangle+\frac{ap^i}{q-1}\right\rfloor=0$ for $0\leq a\leq q-2$ with $a\neq \frac{q-1}{2}$, we obtain
	\begin{align*}
	{_{3}}G_{3}\left[\begin{array}{ccc}
		\frac{1}{4},  &\frac{1}{2}, & \frac{3}{4} \vspace*{0.1cm}\\
		0, & \frac{1}{2}, & \frac{1}{2}
	\end{array}|x
	\right]_q&=-\frac{1}{q-1}\sum_{\begin{subarray}{1} \ a=0, \\a\neq\frac{ q-1}{2} \end{subarray}}^{q-2}(-1)^{3a}\overline{\omega}^a(-x)(-p)^{\sum_{i=0}^{r-1}\alpha_{i,a}}\prod\limits_{i=0}^{r-1}\Gamma_{p}\left(\left\langle\frac{ap^i}{q-1}\right\rangle\right)\\
	&\times\frac{\Gamma_{p}\left(\left\langle\frac{p^i}{4}-\frac{ap^i}{q-1}\right\rangle\right)\Gamma_{p}\left(\left\langle\frac{3p^i}{4}-\frac{ap^i}{q-1}\right\rangle\right)\Gamma_{p}\left(\left\langle\frac{-p^i}{2}+\frac{ap^i}{q-1}\right\rangle\right)}{\Gamma_p\left(\left\langle\frac{p^i}{4}\right\rangle\right)\Gamma_p\left(\left\langle\frac{3p^i}{4}\right\rangle\right)\Gamma_p\left(\left\langle\frac{-p^i}{2}\right\rangle\right)}\\
	&-\frac{\varphi(-x)}{(q-1)(-p)^r}\prod_{i=0}^{r-1}\frac{1}{\Gamma_p\left(\left\langle\frac{p^i}{2}\right\rangle\right)^2},
\end{align*}
where $\alpha_{i,a}=-\left\lfloor\left\langle\frac{p^i}{4}\right\rangle-\frac{ap^i}{q-1}\right\rfloor-\left\lfloor\left\langle\frac{3p^i}{4}\right\rangle-\frac{ap^i}{q-1}\right\rfloor-\left\lfloor\left\langle\frac{-p^i}{2}\right\rangle+\frac{ap^i}{q-1}\right\rfloor-\left\lfloor\frac{ap^i}{q-1}\right\rfloor$. Using Gross-Koblitz formula for $g(\varphi)^2$, we obtain $\prod_{i=0}^{r-1}\Gamma_p\left(\left\langle\frac{p^i}{2}\right\rangle\right)^2=(-1)^r\varphi(-1)$. Adding and subtracting the term under the summation for $a=\frac{q-2}{2}$ and using $\overline{\omega}^a(-1)=(-1)^a$, we deduce that
\begin{align*}
	{_{3}}G_{3}\left[\begin{array}{ccc}
	\frac{1}{4},  &\frac{1}{2}, & \frac{3}{4} \vspace*{0.1cm}\\
	0, & \frac{1}{2}, & \frac{1}{2}
\end{array}|x
\right]_q&=-\frac{1}{q-1}\sum_{a=0}^{q-2}\overline{\omega}^a(x)(-p)^{\sum_{i=0}^{r-1}\alpha_{i,a}}\prod\limits_{i=0}^{r-1}\Gamma_{p}\left(\left\langle\frac{ap^i}{q-1}\right\rangle\right)\\
&\times\frac{\Gamma_{p}\left(\left\langle\frac{p^i}{4}-\frac{ap^i}{q-1}\right\rangle\right)\Gamma_{p}\left(\left\langle\frac{3p^i}{4}-\frac{ap^i}{q-1}\right\rangle\right)\Gamma_{p}\left(\left\langle\frac{-p^i}{2}+\frac{ap^i}{q-1}\right\rangle\right)}{\Gamma_p\left(\left\langle\frac{p^i}{4}\right\rangle\right)\Gamma_p\left(\left\langle\frac{3p^i}{4}\right\rangle\right)\Gamma_p\left(\left\langle\frac{-p^i}{2}\right\rangle\right)}\\
&-\frac{\varphi(x)}{q(q-1)}+\frac{\varphi(x)}{q-1}.
\end{align*}
This completes the proof of the lemma.
\end{proof}
Next, we prove another lemma.
\begin{lemma}\label{lemma_4.4}
Let $p$ be an odd prime and $q=p^r,r\geq1$. We have the following transformations.
\begin{enumerate}
\item For $t\in\mathbb{F}_q^\times$, we have
\begin{align*}
{_2}G_{2}\left[\begin{array}{cc}
				\frac{1}{4},  & \frac{3}{4} \vspace*{0.1cm}\\
				\frac{1}{2}, & \frac{1}{2}
			\end{array}|t
			\right]_q={_2}G_{2}\left[\begin{array}{cc}				 
				\frac{1}{2}, & \frac{1}{2}\vspace*{0.1cm}\\
				\frac{1}{4},  & \frac{3}{4}
			\end{array}|\frac{1}{t}
			\right]_q.
\end{align*}
\item For $t\in\mathbb{F}_q$, we have
\begin{align*}
{_2}G_{2}\left[\begin{array}{cc}
				\frac{1}{4},  & \frac{3}{4} \vspace*{0.1cm}\\
				\frac{1}{2}, & \frac{1}{2}
			\end{array}|t
			\right]_q=\frac{\varphi(-t)}{q}\cdot {_2}G_{2}\left[\begin{array}{cc}
				\frac{1}{4},  & \frac{3}{4} \\
				0, & 0
			\end{array}|t
			\right]_q.
\end{align*}
\item For $t\in\mathbb{F}_q^\times$, we have
\begin{align*}
{_2}G_{2}\left[\begin{array}{cc}
				\frac{1}{3},  & \frac{2}{3} \vspace*{0.1cm}\\
				0, & 0
			\end{array}|t
			\right]_q=\varphi(-3t)\cdot q\cdot {_2}G_{2}\left[\begin{array}{cc}
				\frac{1}{2}, & \frac{1}{2}\vspace*{0.1cm}\\
				\frac{1}{6},  & \frac{5}{6} 
			\end{array}|\frac{1}{t}
			\right]_q.
\end{align*}
\end{enumerate}
Identity $(3)$ is true for $p>3$.
\end{lemma}
\begin{proof}
	We first prove $(1)$. For $t\in\mathbb{F}_q^\times$, we consider
	\begin{align}\label{neweqn-03}
		A_t:&={_2}G_{2}\left[\begin{array}{cc}
			\frac{1}{4},  & \frac{3}{4} \vspace*{0.1cm}\\
			\frac{1}{2}, & 	\frac{1}{2}
		\end{array}|t
		\right]_q\notag\\
		&=\frac{-1}{q-1}\sum_{a=0}^{q-2}\overline{\omega}^a(t)(-p)^{\sum_{i=0}^{r-1}b_{i,a}}\prod_{i=0}^{r-1}\frac{\Gamma_p\left(\left\langle\frac{p^i}{4}-\frac{ap^i}{q-1}\right\rangle\right)\Gamma_p\left(\left\langle\frac{3p^i}{4}-\frac{ap^i}{q-1}\right\rangle\right)}{\Gamma_p\left(\left\langle\frac{p^i}{4}\right\rangle\right)\Gamma_p\left(\left\langle\frac{3p^i}{4}\right\rangle\right)}\notag\\
		&\hspace*{.2cm}\times\frac{\Gamma_p\left(\left\langle\frac{-p^i}{2}+\frac{ap^i}{q-1}\right\rangle\right)^2}{\Gamma_p\left(\left\langle\frac{-p^i}{2}\right\rangle\right)^2},
	\end{align}
where $b_{i,a}=-\left\lfloor\left\langle\frac{p^i}{4}\right\rangle-\frac{ap^i}{q-1}\right\rfloor-\left\lfloor\left\langle\frac{3p^i}{4}\right\rangle-\frac{ap^i}{q-1}\right\rfloor-2\left\lfloor\left\langle\frac{-p^i}{2}\right\rangle+\frac{ap^i}{q-1}\right\rfloor$. Replacing $a$ by $-a$, we obtain the required identity.\\
For $(2)$, we replace $a$ by $a+\frac{q-1}{2}$ in \eqref{neweqn-03} and then using the fact that $$\prod_{i=0}^{r-1}\Gamma_p\left(\left\langle\frac{p^i}{2}\right\rangle\right)^2=(-1)^r\varphi(-1),$$ we obtain the required identity.\\
We now prove $(3)$. For $p>3$ and $t\in\mathbb{F}_q^\times$, we consider
	\begin{align*}
		B_t:&={_2}G_{2}\left[\begin{array}{cc}
			\frac{1}{3},  & \frac{2}{3} \vspace*{0.1cm}\\
			0, & 0
		\end{array}|t
		\right]_q\\
		&=\frac{-1}{q-1}\sum_{a=0}^{q-2}\overline{\omega}^a(t)(-p)^{\sum_{i=0}^{r-1}s_{i,a}}\prod_{i=0}^{r-1}\frac{\Gamma_p\left(\left\langle\frac{p^i}{3}-\frac{ap^i}{q-1}\right\rangle\right)\Gamma_p\left(\left\langle\frac{2p^i}{3}-\frac{ap^i}{q-1}\right\rangle\right)}{\Gamma_p\left(\left\langle\frac{p^i}{3}\right\rangle\right)\Gamma_p\left(\left\langle\frac{2p^i}{3}\right\rangle\right)}\\
		&\hspace*{.2cm}\times\Gamma_p\left(\left\langle\frac{ap^i}{q-1}\right\rangle\right)^2,
	\end{align*}
where $s_{i,a}=-\left\lfloor\left\langle\frac{p^i}{3}\right\rangle-\frac{ap^i}{q-1}\right\rfloor-\left\lfloor\left\langle\frac{2p^i}{3}\right\rangle-\frac{ap^i}{q-1}\right\rfloor-2\left\lfloor\frac{ap^i}{q-1}\right\rfloor$. Replacing $a$ by $a+\frac{q-1}{2}$, we deduce
\begin{align}\label{eq-29}
	B_t&=\frac{-\varphi(t)}{q-1}\sum_{a=0}^{q-2}\overline{\omega}^a(t)(-p)^{\sum_{i=0}^{r-1}t_{i,a}+1}\prod_{i=0}^{r-1}\frac{\Gamma_p\left(\left\langle\frac{p^i}{6}-\frac{ap^i}{q-1}\right\rangle\right)\Gamma_p\left(\left\langle\frac{5p^i}{6}-\frac{ap^i}{q-1}\right\rangle\right)}{\Gamma_p\left(\left\langle\frac{p^i}{3}\right\rangle\right)\Gamma_p\left(\left\langle\frac{2p^i}{3}\right\rangle\right)}\notag\\
	&\hspace*{.2cm}\times\Gamma_p\left(\left\langle\frac{-p^i}{2}+\frac{ap^i}{q-1}\right\rangle\right)^2,
\end{align}
where $t_{i,a}=-\left\lfloor\left\langle\frac{p^i}{6}\right\rangle-\frac{ap^i}{q-1}\right\rfloor-\left\lfloor\left\langle\frac{5p^i}{6}\right\rangle-\frac{ap^i}{q-1}\right\rfloor-2\left\lfloor\left\langle\frac{-p^i}{2}\right\rangle+\frac{ap^i}{q-1}\right\rfloor$. Using the fact that $\prod_{i=0}^{r-1}\Gamma_p\left(\left\langle\frac{p^i}{2}\right\rangle\right)^2=(-1)^r\varphi(-1)$ and Lemma \ref{lemma-3_2} with $t=3$ and $a=\frac{q-1}{2}$, we obtain  
\begin{align}\label{eqn-1}
	\prod_{i=0}^{r-1}\Gamma_p\left(\left\langle\frac{p^i}{3}\right\rangle\right)\Gamma_p\left(\left\langle\frac{2p^i}{3}\right\rangle\right)&=(-1)^r\varphi(-3)\notag\\
	&\hspace*{0.2cm}\times\prod_{i=0}^{r-1}\Gamma_p\left(\left\langle\frac{-p^i}{2}\right\rangle\right)^2\Gamma_p\left(\left\langle\frac{p^i}{6}\right\rangle\right)\Gamma_p\left(\left\langle\frac{5p^i}{6}\right\rangle\right).
\end{align}
Substituting \eqref{eqn-1} in \eqref{eq-29} and replacing $a$ by $-a$, we obtain identity $(3)$.\\
\end{proof}
\begin{cor}\label{cor-5}
	Let $p\geq3$ be a prime and $q=p^r,r\geq1$. 
	\begin{enumerate}
		\item We have
		\begin{align*}
			\sum_{t\in\mathbb{F}_q}\varphi(1-t){_{2}}G_{2}\left[\begin{array}{cc}
				\frac{1}{4},  & \frac{3}{4}\vspace*{0.1cm} \\
				\frac{1}{2}, & \frac{1}{2}
			\end{array}|t
			\right]_q=-\frac{1}{q}- \varphi(2).
		\end{align*} 
	\item Let $x\neq 0, 1$ and $\frac{x-1}{x}$ a square in $\mathbb{F}_q^\times$. If $\frac{x-1}{x}=a^2$ for some $a\in\mathbb{F}_q^\times$, then 
	\begin{align*}
\sum_{t\in\mathbb{F}_q}\varphi(1-t){_{2}}G_{2}\left[\begin{array}{cc}
			\frac{1}{4},  & \frac{3}{4} \vspace*{0.1cm}\\
			\frac{1}{2}, & \frac{1}{2}
		\end{array}|xt
		\right]_q=-\frac{\varphi(x)}{q}-\varphi(2)(\varphi(1+a)+\varphi(1-a)).
	\end{align*}
\item Let $x\neq 0$. If $\frac{x-1}{x}$ is not a square in $\mathbb{F}_q^\times$, then
\begin{align*}
	\sum_{t\in\mathbb{F}_q}\varphi(1-t){_{2}}G_{2}\left[\begin{array}{cc}
		\frac{1}{4},  & \frac{3}{4}\vspace*{0.1cm} \\
		\frac{1}{2}, & \frac{1}{2}
	\end{array}|xt
	\right]_q=-\frac{\varphi(x)}{q}.
\end{align*}
	\end{enumerate}
\end{cor}
\begin{proof}
Taking $d=4$ and $k=2$ in Theorem \ref{MT-4}, we have the following identity
\begin{align*}
	\sum_{t\in\mathbb{F}_q}\varphi(1-t){_{2}}G_{2}\left[\begin{array}{cc}
		\frac{1}{4},  & \frac{3}{4} \vspace*{0.1cm}\\
		\frac{1}{2}, & \frac{1}{2}
	\end{array}|xt
	\right]_q=-{_{3}}G_{3}\left[\begin{array}{ccc}
		\frac{1}{4},  &\frac{1}{2}, & \frac{3}{4}\vspace*{0.1cm} \\
		0,& \frac{1}{2}, & \frac{1}{2}
	\end{array}|x
	\right]_q.
\end{align*}
Using Lemma \ref{lemma_4.3}, we obtain
\begin{align*}
	\sum_{t\in\mathbb{F}_q}\varphi(1-t){_{2}}G_{2}\left[\begin{array}{cc}
		\frac{1}{4},  & \frac{3}{4} \vspace*{0.1cm}\\
		\frac{1}{2}, & \frac{1}{2}
	\end{array}|xt
	\right]_q=-{_{2}}G_{2}\left[\begin{array}{cc}
		\frac{1}{4},  & \frac{3}{4} \vspace*{0.1cm}\\
		0, & \frac{1}{2}
	\end{array}|x
	\right]_q -\frac{\varphi(x)}{q}.
\end{align*}
Using \cite[Theorem 1.2]{NS}, we obtain the required result.
\end{proof}
\begin{remark}
Using Lemma \ref{lemma_4.4} $(2)$ in Corollary \ref{cor-5} and replacing $x$ by $\frac{1}{x}$, we obtain \cite[Theorem 1.7]{NS}.
\end{remark}
\begin{cor}\label{cor-6}
Let $p>3$ be a prime and $q=p^r,r\geq1$. 
\begin{enumerate}
\item We have \begin{align*}
\sum_{t\in\mathbb{F}_q}\varphi(t(t-1)){_{2}}G_{2}\left[\begin{array}{cc}
		\frac{1}{3},  & \frac{2}{3} \\
		0, & 0
	\end{array}|t
	\right]_q=-1-q.
\end{align*}
\item Let $x\neq 0, 1$ be such that $\varphi(3x(1-x))=-1$. Then, we have
\begin{align*}
\sum_{t\in\mathbb{F}_q}\varphi(t(t-1)){_{2}}G_{2}\left[\begin{array}{cc}
		\frac{1}{3},  & \frac{2}{3} \\
		0, & 0
	\end{array}|\frac{t}{x}
	\right]_q=-1.
\end{align*}
\end{enumerate}
\end{cor}
\begin{proof}
Taking $d=3$ and $k=1$ in Theorem \ref{MT-3}, we find that 
\begin{align*}
-\sum_{t\in\mathbb{F}_q}\varphi(t(t-1)){_{2}}G_{2}\left[\begin{array}{cc}
		\frac{1}{3},  & \frac{2}{3} \\
		0, & 0
	\end{array}|\frac{t}{x}
	\right]_q=1+q\cdot{_{2}}G_{2}\left[\begin{array}{cc}
		\frac{1}{3},  &\frac{2}{3}\vspace*{0.1cm} \\
		0, & \frac{1}{2}
	\end{array}|\frac{1}{x}
	\right]_q.
\end{align*}
We obtain $(1)$ and $(2)$ using \cite[Corollary 1.3]{NS1} and \cite[Theorem 1.3]{SB1}, respectively.
\end{proof}
\begin{cor}\label{cor-7}
Let $p>3$ be a prime and $q=p^r,r\geq1$. Let $q\equiv 1\pmod3$ and $\chi_3$ be a character of order $3$. 
\begin{enumerate}
\item We have \begin{align*}
\sum_{t\in\mathbb{F}_q^\times}\varphi(t(t-1)){_{2}}F_{1}\left(\begin{array}{cc}
		\chi_3,  & \overline{\chi_3} \\
		 & \varepsilon
	\end{array}|\frac{1}{t}
	\right)_q=1+\frac{1}{q}.
\end{align*}
\item Let $x\neq 0,1$ be such that $\varphi(3x(1-x))=-1$. Then, we have
\begin{align*}
\sum_{t\in\mathbb{F}_q^\times}\varphi(t(t-1)){_{2}}F_{1}\left(\begin{array}{cc}
		\chi_3,  & \overline{\chi_3} \\
		 & \varepsilon
	\end{array}|\frac{x}{t}\right)_q=\frac{1}{q}.
\end{align*}
\end{enumerate}
\end{cor}
\begin{proof}
	For $x,t\in\mathbb{F}_q^\times$, \cite[Lemma 3.3]{mccarthy2} and \cite[Proposition 2.5]{mccarthy3} yield
	\begin{align*}
		{_{2}}G_{2}\left[\begin{array}{cc}
			\frac{1}{3},  & \frac{2}{3} \\
			0, & 0
		\end{array}|\frac{t}{x}\right]_q
	&={_{2}}F_{1}\left(\begin{array}{cc}
			\chi_3,  & \overline{\chi_3} \\
			& \varepsilon
		\end{array}|\frac{x}{t}\right)_q^*\\
	&=\binom{\overline{\chi_3}}{\varepsilon}^{-1}{_{2}}F_{1}\left(\begin{array}{cc}
		\chi_3,  & \overline{\chi_3} \\
		& \varepsilon
	\end{array}|\frac{x}{t}\right)_q\\
&=-q\cdot{_{2}}F_{1}\left(\begin{array}{cc}
	\chi_3,  & \overline{\chi_3} \\
	& \varepsilon
\end{array}|\frac{x}{t}\right)_q,
	\end{align*}
	where we obtain the last equailty using \eqref{eq-4}. Now,
	\begin{align*}
		\sum_{t\in\mathbb{F}_q}\varphi(t(t-1)){_{2}}G_{2}\left[\begin{array}{cc}
			\frac{1}{3},  & \frac{2}{3} \\
			0, & 0
		\end{array}|\frac{t}{x}\right]_q&=\sum_{t\in\mathbb{F}_q^\times}\varphi(t(t-1)){_{2}}G_{2}\left[\begin{array}{cc}
		\frac{1}{3},  & \frac{2}{3} \\
		0, & 0
	\end{array}|\frac{t}{x}\right]_q \\
&=-q\cdot\sum_{t\in\mathbb{F}_q^\times}\varphi(t(t-1)){_{2}}F_{1}\left(\begin{array}{cc}
	\chi_3,  & \overline{\chi_3} \\
	& \varepsilon
\end{array}|\frac{x}{t}\right)_q.
	\end{align*}
	 Using Corollary \ref{cor-6}, we obtain the desired result.
\end{proof}
\begin{remark}
	Combining \eqref{eqn-1.4} and \eqref{neweqn-1.5} and taking $d=4$ and $k=2$, we obtain
	\begin{align*}
	{_{3}}G_{3}\left[\begin{array}{ccc}
	\frac{1}{4},  &\frac{1}{2}, & \frac{3}{4}\vspace*{0.1cm} \\
	0, & \frac{1}{2}, & \frac{1}{2}
\end{array}|\frac{1}{16\alpha}
\right]_q&=n_q(\alpha)-1+\frac{(1-q)\varphi(\alpha)}{q},
\end{align*}	 
where $n_q(\alpha)$ is the number of distinct zeros of the polynomial $h_\alpha(y)=y^4-2y^3+y^2-\alpha$. We can check that if $\alpha$ is not a square in $\mathbb{F}_q$ then $h_\alpha(y)$ has no zero in $\mathbb{F}_q$ and hence, $n_q(\alpha)=0$. 
Applying Lemma \ref{lemma_4.3}, we deduce that 
	\begin{align*}
		{_{2}}G_{2}\left[\begin{array}{cc}
		\frac{1}{4},  & \frac{3}{4} \vspace*{0.1cm}\\
		0, & \frac{1}{2}
	\end{array}|\frac{1}{16\alpha}
	\right]_q =0,
	\end{align*}
	which partially gives \cite[Theorem 1.2]{NS}.
\end{remark}
We have proved all the required lemmas and corollaries to prove the identities for the trace of Frobenius of elliptic curves. Let $E_1$ and $E_2$ be elliptic curves over $\mathbb{F}_q$ given by
\begin{align*}
	&	E_1:y^2=x^3+fx^2+gx, f\neq 0,\\
	&	E_2:y^2=x^3+ax+b \text{ with } j(E_2)\neq 0, 1728.
\end{align*}
In \cite{mccarthy2}, McCarthy introduced $p$-adic hypergeometric functions and expressed the trace of the Frobenius endomorphism on $E_2$ as a special value of the function ${_2}G_2[\cdots]$ as given in the following theorem.
\begin{theorem}\emph{(\cite[Theorem 1.2]{mccarthy2}).}\label{theorem-4}
	Let $p>3$ be prime. Consider an elliptic curve $E_2/\mathbb{F}_p$ of the form $E_2:y^2=x^3+ax+b \text{ with } j(E_2)\neq 0, 1728.$ Then
	\begin{align*}
	a_p(E_2)=p\cdot\varphi(b)\cdot{_2}G_{2}\left[\begin{array}{cc}
	\frac{1}{4},  & \frac{3}{4} \vspace*{0.1cm}\\
	\frac{1}{3}, & \frac{2}{3}
	\end{array}|\frac{-27b^2}{4a^3}
	\right]_p.
	\end{align*}
\end{theorem}
In \cite{BS1}, the second author with Saikia did the same for the elliptic curve $E_1$.
\begin{theorem}\emph{(\cite[Theorem 3.5]{BS1}).}\label{theorem-3}
	Let $p$ be an odd prime and $q=p^r,r\geq1$. The trace of Frobenius on $E_1$ is given  by
	\begin{align*}
		a_q(E_1)=q\cdot\varphi(-fg)\cdot{_2}G_{2}\left[\begin{array}{cc}
			\frac{1}{2},  & \frac{1}{2}\vspace*{0.1cm} \\
			\frac{1}{4}, & \frac{3}{4}
		\end{array}|\frac{4g}{f^2}
		\right]_q.
	\end{align*}
\end{theorem}
\begin{remark}
McCarthy proved Theorem \ref{theorem-4} for $\mathbb{F}_p$ for all primes $p>3$. In \cite{BS1}, it was verified that Theorem \ref{theorem-4} is true for $\mathbb{F}_q$, where $q=p^r,r\geq1$ and $p>3$.
\end{remark}
\begin{proof}[Proof of Theorem \ref{MT-5}]
	Taking $d=4$ and $k=3$ in Theorem \ref{MT-4}, we have
	\begin{align*}
			\sum_{t\in\mathbb{F}_q}\varphi(1-t){_2}G_{2}\left[\begin{array}{cc}
				\frac{1}{4},  & \frac{3}{4} \vspace*{0.1cm}\\
					\frac{1}{3}, & \frac{2}{3}
			\end{array}|x_0t
			\right]_q&=\sum_{t\in\mathbb{F}_q^\times}\varphi(1-t){_2}G_{2}\left[\begin{array}{cc}
				\frac{1}{4},  & \frac{3}{4} \vspace*{0.1cm}\\
				\frac{1}{3}, & \frac{2}{3}
			\end{array}|x_0t
			\right]_q\\
			&=-{_3}G_{3}\left[\begin{array}{ccc}
			\frac{1}{4},  & \frac{1}{2},& \frac{3}{4} \vspace*{0.1cm}\\
			0, &	\frac{1}{3}, & \frac{2}{3}
		\end{array}|x_0
		\right]_q.
	\end{align*}
Since $q\not\equiv1\pmod3$, therefore $t\mapsto t^{-3}$ is an automorphism on $\mathbb{F}_q^\times$ and we obtain
	\begin{align*}
	\sum_{t\in\mathbb{F}_q^\times}\varphi(t(t^3-1)){_2}G_{2}\left[\begin{array}{cc}
		\frac{1}{4},  & \frac{3}{4} \vspace*{0.1cm}\\
		\frac{1}{3}, & \frac{2}{3}
	\end{array}|\frac{x_0}{t^3}
	\right]_q=-{_3}G_{3}\left[\begin{array}{ccc}
		\frac{1}{4},  & \frac{1}{2},& \frac{3}{4} \vspace*{0.1cm}\\
		0, &	\frac{1}{3}, & \frac{2}{3}
	\end{array}|x_0
	\right]_q.
\end{align*}
Taking $x_0=\frac{-27b^2}{4}$ with $b\in\mathbb{F}_q^\times$ yields
	\begin{align*}
	\sum_{t\in\mathbb{F}_q^\times}\varphi(t(t^3-1)){_2}G_{2}\left[\begin{array}{cc}
		\frac{1}{4},  & \frac{3}{4} \vspace*{0.1cm}\\
		\frac{1}{3}, & \frac{2}{3}
	\end{array}|\frac{-27b^2}{4t^3}
	\right]_q=-{_3}G_{3}\left[\begin{array}{ccc}
		\frac{1}{4},  & \frac{1}{2},& \frac{3}{4} \vspace*{0.1cm}\\
		0, &	\frac{1}{3}, & \frac{2}{3}
	\end{array}|\frac{-27b^2}{4}
	\right]_q.
\end{align*}
Note that for $t,b\in\mathbb{F}_q^\times$, we have $j(E_{t,b})\neq0,1728$. Hence, using Theorem \ref{theorem-4}, we obtain the desired result.
\end{proof}
\begin{proof}[Proof of Theorem \ref{MT-7}]
	We consider a character sum $A_{x_0}$ and then using Lemma \ref{lemma_4.4} (1), we have
	\begin{align}\label{eqn-5.2}
		A_{x_0}&:=\sum_{t\in\mathbb{F}_q}\varphi(1-t){_{2}}G_{2}\left[\begin{array}{cc}
			\frac{1}{4}, & \frac{3}{4}\vspace*{0.1cm}\\
			\frac{1}{2},  & \frac{1}{2} \\
		\end{array}|x_0t
		\right]_q\notag\\
		&=\sum_{t\in\mathbb{F}_q^\times}\varphi(1-t){_{2}}G_{2}\left[\begin{array}{cc}
			\frac{1}{4}, & \frac{3}{4}\vspace*{0.1cm}\\
			\frac{1}{2},  & \frac{1}{2} \\
		\end{array}|x_0t
		\right]_q\notag\\
		&=\sum_{t\in\mathbb{F}_q^\times}\varphi(1-t){_{2}}G_{2}\left[\begin{array}{cc}
			\frac{1}{2},  & \frac{1}{2}\vspace*{0.1cm} \\
			\frac{1}{4}, & \frac{3}{4}\\
		\end{array}|\frac{1}{x_0t}
		\right]_q.
	\end{align}
	Let $E_{f,t}:y^2=x^3+fx^2+\frac{x}{t}$, where $f,t\in\mathbb{F}_q^\times$. Then by Theorem \ref{theorem-3}, we have 
	\begin{align}\label{neweqn-04}
		a_q(E_{f,t})=q\cdot\varphi(-ft)\cdot{_{2}}G_{2}\left[\begin{array}{cc}
			\frac{1}{2},  & \frac{1}{2}\vspace*{0.1cm} \\
			\frac{1}{4}, & \frac{3}{4}\\
		\end{array}|\frac{4}{f^2t}
		\right]_q.
	\end{align}
	Taking $x_0=\frac{f^2}{4}$ in \eqref{eqn-5.2} and using \eqref{neweqn-04}, we have
	\begin{align}\label{neweqn-05}
		A_{\frac{f^2}{4}}=\frac{\varphi(f)}{q}\sum_{t\in\mathbb{F}_q^\times}\varphi(t(t-1))a_q(E_{f,t})
		\end{align}
	Combining \eqref{neweqn-05} and the expression for the character sum $A_{\frac{f^2}{4}}$ from Corollary \ref{cor-5}, we obtain the required result.
\end{proof}
Let $a\in\mathbb{F}_q$ be such that $a^3\neq1$. Then the Hessian curve over $\mathbb{F}_q$ is given by the following cubic equation
\begin{align*}
C_a(\mathbb{\mathbb{F}}_q):x^3+y^3+1=3axy.
\end{align*}
In \cite{BS2}, the second author with Saikia expressed the number of $\mathbb{F}_q$-points on $C_a(\mathbb{F}_q)$ in terms of $_{2}G_{2}[\cdots]$ as given in the following theorem.
\begin{theorem}\emph{(\cite[Theorem 3.3]{BS2}).}\label{theorem-2}
	Let $a\in\mathbb{F}_q^\times$ such that $a^3\neq1$. Let $p>3$ be a prime and $q=p^r,r>1$. Then we have
	\begin{align*}
	\#C_a(\mathbb{F}_q)=\alpha-1+q-q\cdot\varphi(-3a)\cdot{_2}G_{2}\left[\begin{array}{cc}
	\frac{1}{2},  & \frac{1}{2} \vspace*{0.1cm}\\
	\frac{1}{6}, & \frac{5}{6}
	\end{array}|\frac{1}{a^3}
	\right]_q,
	\end{align*}
	where $\alpha=\left\{\begin{array}{cc}
	5-6\varphi(-3), & \hbox{if $q\equiv 1\pmod3$};\\
	1, & \hbox{if $q\not\equiv 1\pmod 3$}.
	\end{array}\right.$
\end{theorem}
\begin{proof}[Proof of Theorem \ref{MT-6}]
Corollary \ref{cor-6} $(1)$	gives
\begin{align*}
	\sum_{t\in\mathbb{F}_q}\varphi(t(t-1)){_{2}}G_{2}\left[\begin{array}{cc}
		\frac{1}{3},  & \frac{2}{3} \\
		0, & 0
	\end{array}|t
	\right]_q&=\sum_{t\in\mathbb{F}_q^\times}\varphi(t(t-1)){_{2}}G_{2}\left[\begin{array}{cc}
		\frac{1}{3},  & \frac{2}{3} \\
		0, & 0
	\end{array}|t
	\right]_q\\
	&=-1-q.
\end{align*}
Using Lemma \ref{lemma_4.4} $(3)$, 
\begin{align*}
	\sum_{t\in\mathbb{F}_q^\times}\varphi(t(t-1)){_{2}}G_{2}\left[\begin{array}{cc}
		\frac{1}{3}, & \frac{2}{3}\vspace*{0.1cm}\\
		0,  & 0\\
	\end{array}|t
	\right]_q&=q\cdot\sum_{t\in\mathbb{F}_q^\times}\varphi(3(1-t)){_{2}}G_{2}\left[\begin{array}{cc}
		\frac{1}{2}, & \frac{1}{2}\vspace*{0.1cm}\\
		\frac{1}{6},  & \frac{5}{6} \\
	\end{array}|\frac{1}{t}
	\right]_q\\
	&=-1-q.
\end{align*}
Since $q\not\equiv1\pmod3$, therefore $t\mapsto t^3$ is an automorphism on $\mathbb{F}_q^\times$. This yields
\begin{align*}
	q\cdot\sum_{t\in\mathbb{F}_q^\times}\varphi(3(1-t^3)){_{2}}G_{2}\left[\begin{array}{cc}
		\frac{1}{2}, & \frac{1}{2}\vspace*{0.1cm}\\
		\frac{1}{6},  & \frac{5}{6} \\
	\end{array}|\frac{1}{t^3}
	\right]_q&=q\cdot\sum_{t\in\mathbb{F}_q^\times,t\neq1}\varphi(3(1-t^3)){_{2}}G_{2}\left[\begin{array}{cc}
		\frac{1}{2}, & \frac{1}{2}\vspace*{0.1cm}\\
		\frac{1}{6},  & \frac{5}{6} \\
	\end{array}|\frac{1}{t^3}
	\right]_q\\
	& = -1-q.
\end{align*}
Let $C_t:x^3+y^3+1=3txy$, where $t\in\mathbb{F}_q^\times$ such that $t\neq1$. Using Theorem \ref{theorem-2}, we obtain
\begin{align*}
	-1-q&=\sum_{t\in\mathbb{F}_q^\times,t\neq1}\varphi(t(t^3-1))(q-\#C_t(\mathbb{F}_q))\\
	&=q\cdot\sum_{t\in\mathbb{F}_q}\varphi(t(t^3-1))-\sum_{t\in\mathbb{F}_q^\times,t\neq1}\varphi(t(t^3-1))\#C_t(\mathbb{F}_q)\\
	&=q\cdot\sum_{t\in\mathbb{F}_q}\varphi(t(t-1))-\sum_{t\in\mathbb{F}_q^\times,t\neq1}\varphi(t(t^3-1))\#C_t(\mathbb{F}_q).
\end{align*}
Using the fact that $\sum_{t\in\mathbb{F}_q}\varphi(t(t-1))=-1$, we obtain the desired result.
\end{proof}


\begin{thebibliography}{99}
	
	
	
	
\bibitem{BK} R. Barman and G. Kalita, {\it Hypergeometric functions over $\mathbb{F}_q$ and traces of Frobenius for elliptic curves}, Proc. Amer. Math. Soc. 141 (2013), no. 10,	3403--3410.
	
\bibitem{BK1}
R. Barman and G. Kalita, {\it Elliptic curves and special values of Gaussian hypergeometric series}, J. Number Theory 133 (2013), 3099--3111.
	
\bibitem{BS2}
R. Barman and N. Saikia, {\it Certain Transformations for Hypergeometric series in the p-adic setting}, Int. J. Number Theory 11 (2015), no. 2, 645--660.
	
\bibitem{BS1}
R. Barman and N. Saikia, {\it $p$-Adic gamma function and the trace of Frobenius of elliptic curves},
J. Number Theory 140 (2014), no. 7, 181--195.
	
\bibitem{BS3} R. Barman and N. Saikia, \textit{Certain character sums and hypergeometric series}, Pac. J. Math. 295 (2018), no. 2, 271--289.


	
\bibitem{BS4} R. Barman and N. Saikia, {\it $p$-adic Gamma function and the polynomials $x^d+ax+b$ and $x^d+ax^{d-1}+b$ over $\mathbb{F}_q$}, Finite Fields and Their Applications, 29 (2014), 89--105.
	
	\bibitem{BS5} 
	R. Barman and N. Saikia, {Summation identities and transformations for hypergeometric series}, Ann. Math. Québec 42 (2018), 133–157.

\bibitem{BSM}
R. Barman, N. Saikia, and D. McCarthy, {\it Summation identities and special values of hypergeometric series in the $p$-adic setting}, J. Number Theory 153 (2015), 63--84.
	
	\bibitem{evans}
	B. Berndt, R. Evans, and K. Williams, {\it Gauss and Jacobi Sums}, Canadian Mathematical Society Series of Monographs and Advanced Texts,
	A Wiley-Interscience Publication, John Wiley \& Sons, Inc., New York, (1998).
	
	
	

	
\bibitem{fuselier}
J.  Fuselier, {\it Hypergeometric  functions  over $\mathbb{F}_p$ and  relations  to  elliptic  curves  and  modular  forms}, Proc. Amer. Math. Soc. 138 (2010), no.1, 109--123.
	
	
	
	
	\bibitem{greene}
	J. Greene, {\it Hypergeometric functions over finite fields}, Trans. Amer. Math. Soc. 301 (1987), no. 1, 77--101.
	
	\bibitem{greene2}
	J. Greene, {\it Character Sum Analogues for Hypergeometric and Generalized Hypergeometric Functions over Finite Fields},
	Ph.D. thesis, Univ. of Minnesota, Minneapolis, 1984.
	
	\bibitem{gross}
	B. H. Gross and N. Koblitz, {\it Gauss sum and the $p$-adic $\Gamma$-function}, Annals of Mathematics 109 (1979), 569--581.
	
	
	
	
	\bibitem{kob} N. Koblitz, {\it $p$-adic analysis: a short course on recent work}, London Math. Soc. Lecture Note Series, 46. Cambridge University Press, Cambridge-New York, (1980).
	
	\bibitem{koike}
	M. Koike, {\it Orthogonal matrices obtained from hypergeometric series over finite fields and elliptic curves over finite fields}, Hiroshima Mathematical Journal 25 (1995), no. 1, 43--52.
	
\bibitem{lennon} C. Lennon, {\it Trace formulas for Hecke operators, Gaussian hypergeometric functions, and the modularity of a threefold}, J. Number Theory 131 (2011), no. 12, 2320--2351.
	
\bibitem{lennon2} C. Lennon, {\it Gaussian hypergeometric evaluations of traces of Frobenius for elliptic curves}, Proc. Amer. Math. Soc. 139 (2011), no. 6, 1931--1938.
	
	
	\bibitem{mccarthy2}
	D. McCarthy, {\it The trace of Frobenius of elliptic curves and the $p$-adic gamma function}, Pac. J. Math. 261 (2013), no. 1, 219--236.
	
	\bibitem{mccarthy3}
	D. McCarthy, {\it Transformations of well-poised hypergeometric functions over finite  fields}, Finite Fields and Their Applications, 18 (2012), no. 6, 1133--1147.
	
	\bibitem{ono} K. Ono, \textit{Values of Gaussian hypergeometric series}, Trans. Amer. Math. Soc. 350 (1998), no. 3, 1205--1223.
	
	
	\bibitem{NS}
	N. Saikia, {\it Zeros of $p$-adic hypergeometric functions, $p$-adic analogues of Kummer’s and Pfaff’s identities}, Pac. J. Math. 307 (2020), no. 2, 491--510.
	
	\bibitem{NS1}
	N. Saikia, {\it Values of $p$-adic hypergeometric functions, $p$-adic analogue of Kummer’s linear identity}, https://arxiv.org/abs/2301.10661.
	
	\bibitem{SB1}
	Sulakashna and R. Barman, {\it Certain transformations and values of $p$-adic hypergeometric functions}, Res. Number Theory (2022) 8:93.
	
	\bibitem{SB}
	Sulakashna and R. Barman, {\it Number of $\mathbb{F}_q$-points on Diagonal hypersurfaces and hypergeometric function}, https://arxiv.org/abs/2210.11732. 
\end{thebibliography}
	\end{document}